\documentclass [12pt]{article}
\usepackage[utf8]{inputenc}
\usepackage[english]{babel}
\usepackage{amsmath}
\usepackage[T2A]{fontenc}
\usepackage{amsfonts, amssymb}
\usepackage{graphicx}
\usepackage[left=2.5cm,right=2.2cm,top=2.5cm,bottom=3.5cm]{geometry}
\usepackage{comment}
\usepackage{cmap}

\newenvironment{proof}[1][] {\noindent {\bf Proof#1:} }{\hspace*{\fill}$\square$\medskip\par}

\newtheorem{thm}{Theorem}
\newtheorem{lem}[thm]{Lemma}
\newtheorem{prop}[thm]{Proposition}
\newtheorem{cor}[thm]{Corrolary}
\newtheorem{example}[thm]{Example}
\newtheorem{rem}[thm]{Remark}

\def\R{{\mathbb R}}
\def\E{{\mathbb E\,}}
\def\P{{\mathbb P}}
\def\N{{\mathbb N}}

\def\eps{\varepsilon}

\def\ow{\overline{w}}
\def\oww{\overline{W}}
\def\orr{\overline{W}^{(r)}}

\DeclareMathOperator*{\argmax}{arg\,max}
\newcommand{\backvec}[1]{\reflectbox{$\vec{\reflectbox{\!$#1$}}$}}

\author{S.E. Nikitin}
\title{On the minimal integral energy of majorants of the Wiener process}	
\date{\today}

\begin{document}

\maketitle

\begin{abstract}
	We consider the asymptotic behavior (over long time intervals) of the minimal integral energy  
\[
	|h|_T^\psi = \int_0^T \psi(h^\prime(t)) \, \mathrm{d}t 
\]  
of majorants of the Wiener process $ W(\cdot) $ satisfying the constraints $ h(0) = r $, $ h(t) \geq W(t) $ for $ 0 \leq t \leq T $.

The results significantly generalize previous asymptotic estimates obtained for the case of kinetic energy $ \psi(u) = u^2 $, revealing that this case, where the minimal energy grows logarithmically, is a critical one, lying between two different asymptotic regimes.

\end{abstract}

\section{Problem statement and main results}

\subsection{Introduction}
Let $AC[0,T]$ be the space of absolutely continuous functions on the interval $[0,T]$.
For an arbitrary convex function $\psi: \mathbb{R} \to \mathbb{R}_{+}$ and any function $h \in AC[0,T]$, consider the energy functional
\[
	|h|_T^\psi := \int_0^T \psi(h^\prime(t)) \mathrm{d}t. 
\]	

The articles \cite{BL20,KL19,IKL19,KL17,LS15,LS21,eS18, LP24, nK22} investigate the problem of approximating trajectories of various stochastic processes by functions $h(\cdot)$ that minimize a certain type of energy, while satisfying constraints on the proximity between $h(\cdot)$ and the process trajectory. Most often, the approximated process is the Wiener process and the energy considered is kinetic energy, that is, $|h|_{T}^{\psi}$ with $\psi(u)=u^2$. The articles \cite{LS15, eS18, LP24} study energy-efficient approximations of the Wiener process $W(\cdot )$ under {\it bilateral} uniform constraints for arbitrary energy functionals. For $T,r>0$, define the set of admissible approximations
\[
	M^\pm_{T, r} := \left\{h \in AC[0, T] \, \big| \, \forall t \in [0, T] : 
	|h(t)-W(t)| \le r; h(0) = 0 \right\}
\]
and the minimal integral energy
\[
		I^{\pm}_W(T,r,\psi) := \inf \left\{|h|_T^\psi \: \big| \: h \in M^\pm_{T,r} \right\}.
\]

In \cite{eS18}, it was shown that the approximation at which $\inf \left\{|h|_T^\psi \: \big| \: h \in M^\pm_{T,r} \right\}$ is attained—called the {\it taut string}—does not depend on the form of $\psi(\cdot)$. 
The result of \cite{eS18} was improved in \cite{LP24}, where it was shown that for a fairly wide class of convex functions $\psi(\cdot)$, and for fixed $r>0$, it holds that
\[
	\frac{I^\pm_W(T,r,\psi)}{T} 
	\stackrel{\mathrm{a.s.}}{\longrightarrow} \int_{\R}\psi (u) \, \nu(\mathrm{d}u), \quad T \to \infty,
\]
where $\nu$ is a measure on $\R$ that depends on $r$ and admits an explicit expression.

In this paper, we are interested in the behavior of the minimum value of the functional  
$|\cdot |_{T}^{\psi}$ over approximations of the Wiener process trajectories under fixed {\it unilateral} constraints; that is, the set of admissible approximations for given $T,r>0$ takes the form
\[
	M_{T, r} := \left\{h \in AC[0, T] \, \big| \, \forall t \in [0, T] : 
	h(t)  \ge  W(t) - r ; h(0) = 0 \right\},
\]
and we are interested in the asymptotic behavior of the minimal integral energy
\[
	I_W^\psi(T,r) := \inf \left\{|h|_T^\psi \: \big| \: h \in M_{T,r} \right\}
\]
as $T \to \infty$.
The asymptotics of this quantity for $\psi(u) = u^2$ is studied in \cite{LN24}.

For technical convenience, we shift the initial value of the approximation to the point $r$, so that the approximating function lies above the trajectory of the process $W(\cdot )$. Define
\[
	M_{T, r}^\prime :=  \left\{h \in AC[0, T] \, \big| \, \forall t \in [0, T] : 
	h(t) \ge W(t); h(0) = r \right\}.
\]
Since the sets $M_{T, r}$ and $M_{T, r}^\prime$ differ only by a constant shift, it is clear that
\[
	I_W^{\psi}(T,r) = \inf \left\{|h|_T^\psi \:\big | \: h \in M_{T,r}^\prime \right\}.
\]
We will show below that, unlike in the bilateral constraint case, the value of $r$ does not affect the asymptotic behavior of $I_W^{\psi}(\cdot ,r)$.

If the minimum of $\psi(\cdot)$ is not attained at zero or is not equal to 0, then $|\cdot |_{T}^{\psi}$ is minimized by a function that is eventually a non-random linear function. Therefore, only functions $\psi(\cdot)$ satisfying $\psi(0)=0$ and attaining their minimum at zero are of interest.

It turns out that the behavior of the minimal energy is determined by the values of the function $\psi(\cdot)$ in a right neighborhood of zero. As follows from Theorem \ref{t:finite} below, if for small $u>0$ the values of $\psi(u)$ are significantly smaller than $u^2$, then the minimal integral energy is finite over the entire time axis; that is,
$\sup_{T>0} I_W^{\psi}(T,r)<\infty$ almost surely. For this reason, the main focus is on convex functions $\psi(\cdot)$ that are $\varkappa$-regularly varying at zero, with $\varkappa \in [1,2]$; that is, functions of the form
\begin{equation} \label{cond:reg}
    \psi(u) = u^{\varkappa}\, \theta(u),
\end{equation}
where $\theta(\cdot)$ is slowly varying at zero. The behavior of $I_W^{\psi}(T,r)$ primarily depends on the value of the parameter $\varkappa$, with the cases $\varkappa=2$ and $\varkappa<2$ being fundamentally different.
To express the conditions imposed on $\psi(\cdot)$, it is convenient to use the Karamata representation for the function $\theta(\cdot)$ (see \cite[Theorem 1.2]{eS85}); namely, for $u \in (0,1)$,
\begin{equation} \label{e:Kara}
    \theta(u) = c(u) \exp \left( \int_{u}^{1} \frac{\varepsilon(v)}{v}\, \mathrm{d}v \right), 
\end{equation} 
where $c(\cdot)$ is a positive bounded measurable function $c(\cdot)$ with a nonzero limit at zero, and $\varepsilon(\cdot)$ is a  bounded measurable function  with zero limit at zero.

In subsection \ref{sec:1.2}, we state results for the case $\varkappa \in [1, 2)$.
In subsection \ref{sec:1.3}, we state results for the case $\varkappa \ge 2$.
\\
In Section \ref{sec:2}, we establish a connection between unilateral energy-efficient approximations of arbitrary continuous functions and their minimal concave majorants. We also derive key properties of the minimal concave majorant of the Wiener process, relying heavily on the results of Groeneboom \cite{pG83}. 
Section \ref{sec:3} contains the proofs of the theorems for the case $\varkappa \in [1, 2)$.  
Section \ref{sec:4} contains the proofs of the theorems for the case $\varkappa \ge 2$.

\subsection{Results for $\varkappa \in [1,2)$} \label{sec:1.2}

The first theorem describes lower bounds for the lower limit of $I_{W}^{\psi}(T,r)$ as $T \to \infty$; specifically, it identifies deterministic functions of $T$ that are almost surely eventually exceeded by $I_W^{\psi}(T,r)$. In other words, any approximation of $W(\cdot)$ almost surely requires more energy to remain above $W(\cdot)$ over the entire time interval $[0,T]$ for all sufficiently large $T$.

\begin{thm} \label{t:liminf_low} 
Let $\psi(\cdot)$ be a convex function of the form \eqref{cond:reg} with $\varkappa \in [1,2)$, and suppose $\theta(\cdot)$ admits the Karamata representation \eqref{e:Kara} for $u \in (0, 1)$. Assume further that there exists $\alpha \in (0,1)$ such that, for all sufficiently small $v > 0$,
\begin{equation} \label{cond:eps_alpha}
       \varepsilon(v) < \frac{1}{|\log v|^{\alpha}}.
\end{equation}
Let $g(\cdot)$ be a slowly varying function at infinity satisfying, for some $\xi_{\circ} > 2- \varkappa > \xi^{\circ} > 0$,
\begin{equation} \label{cond:t1g1}
	(\log b)^{-\xi_{\circ}} \ll g(b) \ll (\log b)^{-\xi^{\circ}}, \quad b \to \infty
\end{equation} 
and
\begin{equation} \label{cond:t1g2}
      \int^{\infty} \frac{g(b)^{1 / (2-\varkappa)}}{b}\, \mathrm{d}b < \infty.
\end{equation} 
Then, for any $r > 0$, it holds that
\begin{equation*} \label{e:liminf_low}
	\liminf_{T \to \infty}\frac{I_{W}^{\psi}(T,r)}{T\psi(T^{-1 / 2})g(T)}    
     \stackrel{\text{a.s.}}{=} \infty.
\end{equation*}

\end{thm}
\begin{rem}
Examples of admissible functions $\theta(\cdot)$ include $|\log(\cdot)|^{\alpha}$ with $\alpha \in \mathbb{R}$ and \, $\exp\left( \pm|\log(\cdot)|^{\alpha} \right)$ with $\alpha \in (0,1)$. Note that the constraint on $\varepsilon(\cdot)$ is unilateral, meaning it controls only the potential growth of $\theta(u)$ as $u \downarrow 0$. The theorem imposes no restrictions on the possible decay of $\theta(u)$ in the same limit. 
\end{rem}

The next theorem describes upper bounds for the lower limit of $I_{W}^{\psi}(T,r)$ as $T \to \infty$; specifically, it identifies deterministic functions of $T$ that almost surely infinitely often exceed $I_W^{\psi}(T,r)$. In other words, these functions demonstrate how small the energy cost of approximating $W(\cdot)$ can be over certain long time intervals, while still remaining above $W(\cdot)$.

\begin{thm} \label{t:liminf_up}
Let $\psi(\cdot)$ be a convex function of the form \eqref{cond:reg} with $\varkappa \in [1,2)$, and suppose $\theta(\cdot)$ admits the Karamata representation \eqref{e:Kara} for $u \in (0, 1)$. Assume further that there exists $\alpha \in (0,1)$ such that, for all sufficiently small $v > 0$,
\[
   |\varepsilon(v)| < \frac{1}{|\log v|^{\alpha}}.
\]
Let $g(\cdot)$ be a slowly varying function at infinity satisfying, for some $\xi_{\circ} > 2- \varkappa > \xi^{\circ} > 0$,
\begin{equation} \label{cond:g3}
	(\log b)^{-\xi_{\circ}} \ll g(b) \ll (\log b)^{-\xi^{\circ}}, \quad b \to \infty
\end{equation} 
and
\begin{equation} \label{cond:g2}
     \int^{\infty} \frac{g(b)^{1 / (2-\varkappa)}}{b}\, \mathrm{d}b = \infty.
\end{equation} 
Then, for any $r > 0$, it holds that
\begin{equation*} \label{e:liminf_up}
	\liminf_{T \to \infty}\frac{I_{W}^{\psi}(T,r)}{T\psi(T^{-1 / 2})g(T)}    
     \stackrel{\text{a.s.}}{=} 0.
\end{equation*}

\end{thm}

\begin{rem}
Examples of admissible functions $\theta(\cdot)$ include $|\log(\cdot)|^{\alpha}$ with $\alpha \in \mathbb{R}$ and \, $\exp\left( \pm |\log(\cdot)|^{\alpha} \right)$ with $\alpha \in (0,1)$. In contrast to Theorem \ref{t:liminf_low}, the constraint on $\varepsilon(\cdot)$ in this case is bilateral.   
\end{rem}

In the case $\psi(u) = u$, Theorems \ref{t:liminf_low} and \ref{t:liminf_up} together yield the well-known Hirsch integral test \cite{eC78, wH65}. 
\begin{cor} \label{cor:Hirsch}
For any slowly varying function $g(\cdot)$ at infinity satisfying condition \eqref{cond:g3} with $\varkappa = 1$, it holds that
\[
    \liminf_{T \to \infty}\frac{\max_{0 \le t \le T} W(t)}{T^{1/2} g(T)} 
    \stackrel{\text{a.s.}}{=} 
    \begin{cases}
	0, &   \text{if }   \int^{\infty} \frac{g(b)}{b}\, \mathrm{d}b = \infty, \\
	\infty, &   \text{if }  \int^{\infty} \frac{g(b)}{b}\, \mathrm{d}b < \infty.
    \end{cases} 
\] 
\end{cor}
\medskip

We now describe the behavior of the {\it upper} limit of $I_W^{\psi}(T,r)$ as $T \to \infty$. The following theorem provides maximal necessary energy values required for approximating $W(\cdot)$ from some sufficiently large time onward, and characterizes  energy levels that may be insufficient.

\begin{thm} \label{t:limsup}
Let $\psi(\cdot)$ be a convex function of the form \eqref{cond:reg} with $\varkappa \in [1,2)$, and suppose $\theta(\cdot)$ admits the Karamata representation \eqref{e:Kara} for $u \in (0, 1)$. Assume further that there exists $\alpha \in (0,1)$ such that, for all sufficiently small $v > 0$,
\[
   \varepsilon(v) > - \frac{1}{|\log v|^{\alpha}}.
\]
Then, for any $r > 0$, it holds almost surely that
\[
   \frac{2^{-\varkappa / 2}}{1 - \varkappa / 2}
	\ge \limsup\limits_{T \to \infty} 
   \frac{I_{W}^{\psi}(T,r)}{T\psi \left( \sqrt{ \frac{\log \log T}{T} } \right) } 
   \ge 2^{\varkappa / 2}. 
\]
\end{thm}
\begin{rem}
Note that the constraint on $\varepsilon(\cdot)$ is unilateral, regulating only the possible decay of $\theta(u)$ as $u \downarrow 0$. The theorem imposes no restrictions on the possible growth of $\theta(\cdot)$ in the same limit. 
\end{rem}
\begin{example}
For $\psi(u) = u^{\varkappa}$ with $\varkappa \in [1, 2)$ and any $r > 0$, it holds almost surely that
\[
	\frac{2^{-\varkappa / 2}}{1 - \varkappa / 2} \ge \limsup\limits_{T \to \infty} \frac{I_{W}^{\psi}(T,r)}{T^{1-\varkappa/2} (\log\log T)^{\varkappa/2} }\ge  2^{\varkappa / 2}. 
\] 
\end{example}

\subsection{Results for $\varkappa \ge 2$} \label{sec:1.3}
	
In contrast to the case $\varkappa \in [1,2)$, when $\varkappa \ge 2$, the minimal energy $I_W^{\psi}(T,r)$ is almost surely either bounded on the entire positive real axis or exhibits an unbounded deterministic asymptotic growth.
The key quantity in this alternative is	
\[
	m_{\psi}(b) := \int^{b}_{1} a \psi(a^{-1}) \mathrm{d}a, \quad b \ge 1.
\] 
The conditions that determine the behavior of $I_{W}^{\psi}(T,r)$ are
\begin{eqnarray}		\label{cond:finite}
		\lim_{b \to \infty}  m_{\psi}(b) < \infty, 
\\		\label{cond:infinite}
	\lim_{b \to \infty} m_{\psi}(b) = \infty.
\end{eqnarray}
Due to the monotonicity of $m_{\psi}(\cdot)$, exactly one of these conditions holds.

\subsubsection*{Finite energy case}  

It turns out that the boundedness of $m_{\psi}(\cdot)$ implies the boundedness of 
$I_{W}^{\psi}(\cdot ,r)$.
	
\begin{thm} \label{t:finite}
Let $r > 0$, and let $\psi(\cdot)$ be a convex function attaining its minimal value 0 at zero. If condition \eqref{cond:finite} holds, then, almost surely, the random variable $I_{W}^{\psi}(T,r)$ is uniformly bounded in $T$ and increases monotonically to its finite limit as $T \to \infty$.
\end{thm}
	\begin{example}
The functions $\psi(u)=u^{\varkappa}$ with $\varkappa > 2$, as well as $\psi(u) = u^2 |\log u|^{\alpha}$ with $\alpha < -1$, satisfy condition \eqref{cond:finite}, along with functions $\psi(\cdot)$ that decay to 0 faster as $u \downarrow 0$.	
	\end{example}	

\subsubsection*{Exact asymptotics case}      
In the case where $m_\psi(\cdot)$ is unbounded, the exact asymptotic growth of $I_{W}^{\psi}(\cdot ,r)$ can be determined.

\begin{thm} \label{t:infinite}
Let $\psi(\cdot)$ be a convex function of the form \eqref{cond:reg} with $\varkappa = 2$ and some function $\theta(\cdot)$, which is not necessarily slowly varying. Suppose that for some $\alpha^{\diamond}, \alpha_\diamond$, one of the following conditions holds: 
\begin{equation} \label{cond:alpha}
	-\frac{1}{2} < \alpha^{\diamond}, \quad \alpha_\diamond \le \alpha^{\diamond} < \alpha_\diamond+ \frac{1}{3},
\end{equation}
or
\begin{equation} \label{cond:alpha_alt}
	-1 < \alpha^{\diamond} \le -\frac{1}{2}, \quad \alpha_\diamond\le \alpha^{\diamond} < 3\alpha_\diamond+2.
\end{equation}
Assume that
\begin{equation} \label{cond:theta}
    |\log u|^{\alpha_{\diamond}} \ll \theta(u) \ll |\log u|^{\alpha^{\diamond}}, \quad u \to 0.
\end{equation}
If condition \eqref{cond:infinite} is satisfied, then for any $r > 0$, it holds that 
		\[
		\lim\limits_{T \to \infty} \frac{I_W^{\psi}(T,r)}{m_{\psi}\left(T^{1/2}\right)} \stackrel{\text{a.s.}}{=} 1 \, . 
		\]	
	\end{thm}

\begin{example}
The functions $\psi(u)=u^2|\log u|^{\alpha} \, \Upsilon(|\log u|)$ with $\alpha > -1$ and some slowly varying function $\Upsilon(\cdot)$ at infinity satisfy condition \eqref{cond:infinite}. For $\Upsilon(\cdot) \equiv 1$, by Theorem \ref{t:infinite}, for any $r > 0$, it holds that
\[
   \lim_{T \to \infty}\frac{I_{W}^\psi(T,r)}{\left( \log T \right)^{1+\alpha}} \stackrel{\text{a.s.}}{=} \frac{1}{(1+\alpha)2^{1+\alpha}}. 
\] 
The special case $\alpha=0$ and $\Upsilon(\cdot) \equiv 1$ (kinetic energy) is considered in \cite{LN24}.
\end{example}

%%%%%%%%%%%%%%%%%%%%%%%%%%%%%%%%%%%%%%%%%%%%%%%%%%%%%%%%%%%%%%%%%

\section{Preliminaries} \label{sec:2}
\subsection{Concave majorants as energy-efficient approximations} \label{subsec:aux1}

The energy-efficient approximation under a unilateral constraint can be characterized in terms of the minimal concave majorant (MCM) of the approximated function and does not depend on $\psi(\cdot)$.  
Let $w:[0,T]\mapsto \R$ be a continuous function. Then the corresponding  
MCM $\ow(\cdot)$ is the smallest concave function satisfying the condition
\[
\ow (t) \ge w(t), \qquad 0\le t\le T.
\]
\begin{prop} \label{p:maj_form} 
Let $r > w(0)$ and let $\psi(\cdot)$ be a convex function attaining its minimal value 0 at zero. Assume that, for each $\delta > 0$, $\ow(\cdot)$ is a piecewise linear function with a finite number of non-decreasing segments on $[\delta, T]$. Then the problem $|h|_T^\psi\to \min$ under the constraints
$h(0)=r$ and
\[
     h(t) \ge w(t), \qquad 0\le t\le T,
\]
admits a solution $\chi^{\text{ }}_*(\cdot)$ of the following form.

(a) If $r \ge \max_{0\le t\le T} w(t)$, then $\chi^{\text{ }}_*(t)\equiv r$.

(b) If $r < \max_{0\le t\le T} w(t)$, then $\chi^{\text{ }}_*(\cdot)$ is defined differently on three segments. On the initial segment, $\chi^{\text{ }}_*(\cdot)$ is an affine function whose graph passes through the point $(0,r)$ and is tangent to the graph of $\ow(\cdot)$.  
Then $\chi^{\text{ }}_*(\cdot)$ coincides with $\ow(\cdot)$ until the first time the maximum of $w(\cdot)$ is attained.  
Finally, after this point, $\chi^{\text{ }}_*(\cdot)$ remains constant.

Moreover, if $\psi(\cdot)$ is strictly convex, then this solution is unique.
\end{prop}

\begin{proof}[ of the proposition]
Case (a) is trivial, so consider case (b).  
Suppose there exists ${\chi^{\text{ }}_{\circ}}(\cdot)$ such that $| \chi^{\text{ }}_{\circ} |_T^\psi < | \chi^{\text{ }}_{*}  |_T^\psi  $. Define the function
\[
\chi(t):= r +\int_0^t \vec{\chi}^{\text{ }\prime}_{\circ}(\vartheta) \mathrm{d}\vartheta, \qquad 0\le t \le T,
\]
where $\vec{\chi}^{\text{ }\prime}_{\circ}(\cdot)$ is the decreasing rearrangement of the function $\max\{\chi^{\prime}_{\circ}(\cdot),0\}$.  
By construction, $\chi(\cdot)$ is a concave non-decreasing function, with $\chi(0)=r$ and 
$\chi(t) \ge \chi^{\text{ }}_{\circ}(t) \ge w(t)$ for all $t\in[0,T]$. Therefore, $\chi(\cdot)$ satisfies the problem's constraints.  
On the other hand,
\[
|\chi|_T^\psi=\int_0^T \psi( \vec{\chi}^{\text{ }\prime}_{\circ}) \mathrm{d}t = \int_0^T  \psi(\max\{\chi_{\circ}^{\prime}(t),0\}) \mathrm{d}t 
\le  |\chi^{\text{ }}_{\circ}|_T^{\psi} < | \chi^{\text{ }}_{*} |_T^\psi. 
\]
Since $\chi^{\text{ }}_*(\cdot)$ is the minimal convex non-decreasing function satisfying the constraints of the problem, we have
\[
\chi(t)\ge \chi^{\text{ }}_*(t), \qquad 0\le t\le T.
\]
Let $k_{*}$ be the number of affine components composing $\chi^{\text{ }}_{*}(\cdot)$, finite by assumption.  
Let $\{f_{k}(\cdot)\}_{k=1}^{k_{*}}$ be affine functions on $[0,T]$ such that
\[
\chi^{\text{ }}_{*}(t) = \min_{1 \le k \le k_{*}}\{f_{k}(t)\}, \qquad 0\le t\le T.
\]
For $k_{*} \ge k \ge 1$ and $t \in [0,T]$, denote $\chi^{\text{ }}_0(t):= \chi(t)$ and 
\[
\chi^{\text{ }} _{k}(t) := \min\left\{ \min_{1 \le i \le k}\{f_{i}(t)\}, \chi(t) \right\}.
\]
For $k_{*} > k \ge 0$, we have 
\[
\chi^{\text{ }} _{k+1}(t) = \min \left\{ \chi^{\text{ }}_{k}(t), f_{k+1}(t) \right\}, \qquad 0\le t\le T.
\]
We will prove that for $k_{*} > k \ge 0$, it is true that
\[
|\chi^{\text{ }}_{k + 1}|_T^\psi \le |\chi^{\text{ }}_{k}|_T^\psi.
\]
Let $[t_{1},t_{2}] \subset [0,T]$ be the maximal interval on which $\chi^{\text{ }}_{k}(\cdot) \ge f _{k+1}(\cdot)$. Due to the convexity of $\chi^{\text{ }}_{k}(\cdot)$, this interval is well-defined.

Using Jensen's inequality, we obtain
\begin{eqnarray*}
			&& (t_2-t_1) \int_{t_1}^{t_2} \frac{1}{t_2-t_1}\psi(\chi^{\prime}_{k}(t)) \mathrm{d}t
            \ge 
			(t_2-t_1)\psi \left(\int_{t_1}^{t_2} \frac{1}{t_2-t_1} \chi^{\prime }_{k}(t)\mathrm{d}t\right)
\\
			&=& (t_2-t_1)\psi \left(\frac{\chi^{\text{ }}_{k}(t_2)-\chi^{\text{ }}_{k}(t_1)}{t_2-t_1}\right)
			\ge   (t_2-t_1)\psi 
                \left(\frac{f_{k+1}(t_2)-f_{k+1}(t_1)}{t_2- t_1}\right) 
\\
			&=&  \int_{t_1}^{t_2} \psi(f_{k+1}'(t)) \mathrm{d}t.
\end{eqnarray*}
Thus, we have
\[
		\int_{t_1}^{t_2} \psi(\chi^{\prime}_{k}(t)) \mathrm{d}t \ge  \int_{t_1}^{t_2} \psi(f_{k+1}'(t)) \mathrm{d}t.		
\]
It follows that $\chi^{\text{ }}_{k+1}(\cdot)$ satisfies the problem's constraints and  
$|\chi^{\text{ }}_{k+1}|_T^\psi \le  |\chi^{\text{ }}_k|_T^\psi$.
Hence, we obtain that $| \chi |_T^\psi \ge | \chi^{\text{ }}_{k_{*}} |_T^\psi = | \chi^{\text{ }}_{*} |_T^\psi $,  
which leads to a contradiction.  

If $\psi(\cdot)$ is strictly convex, then the energy functional $| \cdot |_T^\psi$ is also strictly convex, ensuring that $\chi^{\text{ }}_{*}(\cdot)$ is the unique minimizer of $| \cdot |_T^\psi$.
\end{proof}

%%%%%%%%%%%%%%%%%%%%%%%%%%%%%%%%%%%%%%%%%%%%%%%%%%%%%%%%%%%%%%%%%%%

\subsection{Simplified representation of the minimal integral energy}
For any function $f: \R \to \R$ and any set $F \subset \R$, define
\[
	\argmax_{t \in F} f(t) :=  \sup \left\{t_0 \in F \, \big| \, f(t_0)=\sup_{t \in F}f(t)\right\}.
\] 
Consider the stochastic process  
	\begin{equation} \label{def:tau}
	\tau(a) := \argmax_{t \ge 0}\, (W(t) - t/a), \qquad a > 0.
	\end{equation}
According to \cite[Theorem~2.1]{pG83}, it admits the following representation 
\begin{equation} \label{e:tau_int}
\tau(a) = \int\limits_{0+}^{\infty} l \, N([0,a]\times \mathrm{d}l), \qquad a > 0,
\end{equation}
where $N(\mathrm{d}a\times \mathrm{d}l)$ is a Poisson random measure with intensity measure 
\begin{equation} \label{e:mu_def}
\mu(\mathrm{d}a\times \mathrm{d}l) = (a^{2}\sqrt{ l })^{-1}\varphi(\sqrt{ l }/a)\, \mathrm{d}a \, \mathrm{d}l, \qquad a>0, \, l>0,
\end{equation}
and $\varphi(\cdot )$ denotes the standard normal density function.
\\
We now present several estimates for $\tau(\cdot)$ that will be used later. 
	
\begin{lem} \label{l:tau_est}
		Let $\delta > 0$, and let $\rho(\cdot )$ be a slowly varying function at infinity, such that $\rho(b) < 1$ for every $b > 0$ and
	\begin{equation} \label{cond:rho}
	\int^{\infty} \frac{\rho(b)^{1 / 2}}{b} \mathrm{d}b < \infty.
	\end{equation} 
	Then, almost surely, for all sufficiently large $a$, it holds that
		\begin{equation*} 
			(2+\delta)a^2\log \log a > \tau(a) > a^2 \rho(a).
		\end{equation*}
	\end{lem}
	\begin{proof}
	According to \cite[Corollary~2.1]{pG83}, for any $a>0$ the variable 
	$\frac{\tau(a)}{a^2}$ has a probability density
	\[
	p(y) = 2 \, \E\left(\frac{X}{\sqrt{y}} - 1\right )_+, \qquad y>0,
	\]
	where $ x_+ := \max(x,0)$ and $ X $ is a standard normal random variable.

	{\bf Upper bound}. For $\log \log a > 1$ and any $\epsilon > 0$, we have
		\begin{eqnarray*}
			\P\left( \frac{\tau ( a )}{a^2} \ge (2+\epsilon)\log\log a \right)
			&=& \int_{(2+\epsilon)\log\log a}^\infty 2\,\E\left(\frac{X}{\sqrt{y}} - 1\right )_+ \mathrm{d}y
			\\
			&=& \sqrt{\frac{2}{\pi}} \int_{(2+\epsilon)\log\log a}^\infty \int_{\sqrt y}^\infty 
			\left(\frac x {\sqrt y} - 1\right )e^{-x^2/2} \, \mathrm{d}y \mathrm{d}x
			\\
			&\le& \sqrt{\frac{2}{\pi}}\int_{(2+\epsilon)\log\log a}^\infty e^{-y/2} y^{-1/2} \mathrm{d}y
			\\ &\le& \sqrt{\frac{8}{\pi}} \, e^{-(1+\epsilon / 2)\log\log a}.
		\end{eqnarray*}
		Let $ a_n := (1+\delta_1)^n $ with some $\delta_1 > 0$, and define the events
    \[
         D_n := \left\{\tau(a_{n}) \ge (2+\delta_1)a_n^2 \log\log a_n \right \}. 
    \]
Then the series $\sum_{n=1}^{\infty}   \mathbb{P} (D_{n})$ converges. Hence, by the Borel--Cantelli lemma, almost surely, for all sufficiently large $n$, we have
\begin{equation*} \label{upest}
			\tau(a_{n}) < (2+\delta_1)a_n^2 \log\log a_n.
\end{equation*}
As $n \to \infty$, for any $ a \in [a_{n}, a_{n+1}] $, using that $\tau(\cdot)$ is non-decreasing, we obtain  
\begin{eqnarray*}
			\tau(a) \le \tau(a_{n+1}) 
			&<& (2+\delta_1)a_{n+1}^2 \log\log a_{n+1}
\\
			&=& (2+\delta_1)(1+\delta_1)^2(a_{n}^2\log\log a_{n})(1+o(1)) 
\\              &\le& (2+\delta_1)(1+\delta_1)^3 \, a^2\log\log a.
\end{eqnarray*}
Choosing $\delta_1>0$ so that $(2+\delta_1)(1+\delta_1)^{3} = 2+\delta$, 
we obtain the required upper bound. 

{\bf Lower bound.} Similarly to the upper bound, for $a > 0$, we have 
\begin{eqnarray*}
			\P\left( \frac{\tau(a) }{a^2}\le \rho(a) \right) 
			&=& 
			\int_0^{\rho(a)} 2\, \E\left(\frac{X}{\sqrt{y}} - 1\right )_+ \mathrm{d}y
\\
			&\le& \sqrt{\frac{2}{\pi}} \int^{\rho(a)}_0 e^{-y/2} y^{-1/2} \mathrm{d}y
			\le \sqrt{\frac{8}{\pi}}  \rho(a)^{1/2}.
\end{eqnarray*}
Let $a_{n}:=2^{n}$ and 
$D_n := \{	\tau(a_{n}) \le  a_{n}^2 \, \rho(a_{n}) \}$.
Since $\rho(\cdot )$ is a slowly varying function that satisfies condition \eqref{cond:rho}, it is true that
\[
			\sum_{n=1}^{\infty} \mathbb{P} (D_{n}) \le \sum_{n=1}^{\infty} \sqrt{\frac{8}{\pi}} \rho(a_{n})^{1 / 2} < \infty.  	
\]
By the Borel--Cantelli lemma, almost surely, for all sufficiently large $n$, we have
\begin{equation*} 
			\tau(a_{n})> a_n^2 \rho(a_{n}).
\end{equation*}
As $n \to \infty$, for any $ a \in [a_{n}, a_{n+1}] $, using that $\tau(\cdot)$ is non-decreasing, we obtain 
\begin{equation*}
			\tau(a) \ge \tau(a_{n}) 
			> a_{n}^2\rho(a_{n})  
			= \frac{a_{n+1}^2}{4}\rho(a_{n+1})(1+o(1))
			\ge \frac{a^2}{8}\rho(a).
\end{equation*}
In this inequality, the function $\rho(\cdot)$ can be replaced with $8\rho(\cdot)$, which gives the required lower bound. 
\end{proof}

\vspace{5mm}

Let $\oww(\cdot)$ denote the global MCM of the standard Wiener process $W(\cdot)$ constructed on the entire half-line $[0,\infty)$, starting from height 0. Define $\oww^{\, \prime}(t)$ as the derivative of $\oww(\cdot)$ at all points $t > 0$ where the derivative exists. At the remaining points, which form at most a countable set due to the convexity of $\oww(\cdot)$, we extend $\oww^{\, \prime}(\cdot)$ by left-continuity. The behavior of $\tau(\cdot)$ determines the properties of $\oww(\cdot)$ and $\oww^{\, \prime}(\cdot)$, as it is true that
	\begin{equation} \label{S_tau}
		\oww^{\, \prime}(t) \ge \frac{1}{a} \iff \tau(a) \ge t.
	\end{equation}
For this and other properties of processes $\tau(\cdot)$ and $\oww^{\, \prime}(\cdot)$, see \cite{pG83}.

For an arbitrary convex function $\psi(\cdot)$, define the key auxiliary object
\[
      L_{\psi}(b_{\circ},b) := \int_{(b_{\circ}, b]} \psi(a^{-1}) \, \mathrm{d}\tau(a), \quad b \ge b_\circ \ge 1.
\]
We will often use $L_{\psi}(\cdot, \cdot)$ with the first argument fixed at $b_{\circ} = 1$; namely
\[
	L_{\psi}(\cdot) := L_{\psi}(1, \cdot).
\]
It turns out that $L_{\psi}(\cdot)$ contains a significant amount of information about 
$I_W^{\psi}(\cdot,r)$.

\begin{prop} \label{p:L_S}
Let $\psi(\cdot)$ be a convex function. Then for any $b \ge b_{\circ} \ge 1$, we have
\begin{equation*}
L_{\psi}(b_{\circ},b) = \int_{\tau(b_{\circ})}^{\tau(b)} \psi(\oww^{\, \prime}(t)) \, \mathrm{d}t.
\end{equation*}
\end{prop}

\begin{proof}
Let $\psi_n(\cdot)$ be a step function such that, for some $k_n \in \mathbb{N}$ and a collection of values
$\{ u_i \in [b^{-1}, b_{\circ}^{-1}], \, x_i \in \mathbb{R} \mid 0 \le i \le k_n \}$, it holds that
\[
\psi_n(u) = \sum_{i=0}^{k_n -1} x_i \cdot 1_{[u_i, u_{i+1})}(u), \quad u \in [b^{-1}, b_{\circ}^{-1}).
\]
Then, using \eqref{S_tau}, we obtain
\[ 
\begin{aligned}
	\int\limits_{\tau(b_{\circ})}^{\tau(b)} \psi_{n}(\oww^{\, \prime}(t)) \, \mathrm{d}t &= \sum_{i=0}^{k_{n}-1} x_{i} \cdot  \Big|\left\{u_{i+1} > \oww^{\, \prime}(t) \ge u_{i}\right\}\Big| = \sum_{i=0}^{k_{n}-1}   x_{i} \cdot \big|\left\{\tau(u_{i}^{-1}) \ge t > \tau(u_{i+1}^{-1})\right\}\big| \\
	&= 
\sum_{i=0}^{k_{n}-1} x_{i} \cdot  \left(\tau(u_{i}^{-1}) - \tau(u_{i+1}^{-1}) \right) =
\int_{(b_\circ,b]} \psi_{n}(a^{-1}) \, \mathrm{d} \tau(a).  
\end{aligned}
\]
Taking the uniform limit $\psi_n(\cdot) \to \psi(\cdot)$ in both integrals yields the desired equality.
\end{proof}

\begin{rem}
In \cite{pG83}, this equality is used for $\psi(u) = u^{2}$.
\end{rem}

The following proposition facilitates the transfer of results from $L_{\psi}(\cdot)$ to $I_W^{\psi}(\cdot, r)$.

\begin{prop} \label{p:L_I}
Let $\psi(\cdot)$ be a convex function attaining its minimal value 0 at zero. Then, for any $r > 0$, the random variable
\[
L_{\psi}(b) - I_W^{\psi}(\tau(b), r), \quad b > 0
\]
is almost surely eventually constant as $b \to \infty$.
\end{prop}

\begin{proof}
Define the MCM $\orr(\cdot)$ of the Wiener process $W(\cdot)$ on the entire half-line $\mathbb{R}_+$, starting from height $r$, as the MCM of $W(t), t \ge 0$, subject to the additional constraint $\orr(0) = r$.
Then, on some initial interval $[0, \mathcal{T}(r)]$, the minorant $\orr(\cdot)$ is an affine function whose graph passes through the point $(0, r)$ and is tangent to the graph of $\oww(\cdot)$, while on $[\mathcal{T}(r), \infty)$, $\orr(\cdot)$ coincides with $\oww(\cdot)$.

Consider the restriction of $\orr(\cdot)$ to $[0, \tau(b)]$, which is a local majorant of $W(\cdot)$ on $[0, \tau(b)]$ starting from height $r$. Moreover, by Proposition \ref{p:maj_form}, it solves the problem
\[
|h|_{\tau(b)}^\psi \to \min, \quad h \in M'_{\tau(b), r}.
\]
This proposition is applicable because the number of jumps of $\tau(\cdot)$ on any interval bounded away from zero and infinity is almost surely finite, and thus, by \eqref{S_tau}, the number of increasing linear segments of $\orr(\cdot)$ on such an interval is also finite.
Since for large $b$ we have $\max_{0 \le t \le \tau(b)} W(t) > r$, part (b) of that proposition applies in this case.

Compare the expressions
\[
I_W^{\psi}(\tau(b), r) = |\orr|_{\tau(b)}^\psi = \psi\big((\orr)'(0)\big) \mathcal{T}(r) + \int_{\mathcal{T}(r)}^{\tau(b)} \psi(\oww^{\, \prime}(t)) \, \mathrm{d}t, \quad \tau(b) \ge \mathcal{T}(r),
\]
and
\[
L_{\psi}(b) = \int_{\tau(1)}^{\tau(b)} \psi(\oww^{\, \prime}(t)) \, \mathrm{d}t,
\]
where the second equality follows from Proposition \ref{p:L_S}. These random variables differ by a term independent of $b$.
\end{proof}

\subsection{Integral Representations of $L_{\psi}(\cdot, \cdot)$}

Using \eqref{e:tau_int} and \eqref{e:mu_def}, we obtain the following integral representation of $L_{\psi}(\cdot, \cdot)$:
\begin{equation} \label{e:L_int}
       L_{\psi}(b_{\circ}, b) =  \int\limits_{0+}^{\infty} \int\limits_{b_{\circ}+}^{b} l \, \psi(a^{-1}) \, 
       N(\mathrm{d}a \times \mathrm{d}l), \quad b \ge b_{\circ} \ge 1.
\end{equation}
Making the change of variables $z = \frac{\sqrt{l}}{a}$, $s = \frac{a}{b}$ in the representation \eqref{e:L_int}, we obtain  
\begin{equation} \label{e:L_int2}  
	\frac{L_{\psi}(b_{\circ}, b)}{b^{2} \psi(b^{-1})} = \int\limits_{0+}^{\infty} \int\limits_{b_{\circ}/b \,+}^{1} z^{2} s^{2-\varkappa} \frac{\theta((sb)^{-1})}{\theta(b^{-1})} \, 
       N_{1}(\mathrm{d}s \times \mathrm{d}z),
\end{equation}
where $N_1(\mathrm{d}s \times \mathrm{d}z)$ is a Poisson random measure with intensity measure
\[
\mu_1(\mathrm{d}s \times \mathrm{d}z) := 2 \varphi(z) s^{-1} \mathrm{d}s \mathrm{d}z, \qquad s \in (b^{-1}, 1], \, z > 0.
\]
The integral representation of $L_{\psi}(\cdot, \cdot)$ with this normalization will be used in the proofs of Theorems \ref{t:liminf_low} and \ref{t:liminf_up}.

Let $\mu_{\psi, b_{\circ}, b}(\mathrm{d}\ell)$ denote the image measure of $\mu_1(\mathrm{d}s \times \mathrm{d}z) \big|_{\{ s \in (b_{\circ}/b,\,1], \, z > 0 \}}$ under the mapping $(s,z) \mapsto z^2 s^{2-\varkappa} \frac{\theta((sb)^{-1})}{\theta(b^{-1})}$. We obtain the following representation
\begin{equation} \label{e:L_int3}
  \frac{L_\psi(b_{\circ}, b)}{b^{2} \psi(b^{-1})} = \int_{0+}^\infty \ell \, N_{\psi, b_{\circ}, b}(\mathrm{d}\ell),
\end{equation}
where $N_{\psi, b_{\circ}, b}(\mathrm{d}\ell)$ is a Poisson random measure with intensity measure $\mu_{\psi, b_{\circ}, b}(\mathrm{d}\ell)$.

We will omit the index $b_{\circ}$ in the notation of Poisson random measures and their intensity measures corresponding to the $L_{\psi}(\cdot)$; that is, for any $b \ge 1$,
\[
\mu_{\psi, b} := \mu_{\psi, 1, b}, \quad N_{\psi, b} := N_{\psi, 1, b}.
\]

\section{Proofs of the theorems for $\varkappa \in [1,2)$} \label{sec:3}
\subsection{Proof of Theorem \ref{t:liminf_low}}

For any $\xi_{\star} > 2 - \varkappa > \xi^{\star} > 0$ and $b > 1$, define the set
\begin{equation} \label{def:P}
P(\xi_{\star}, \xi^{\star}, b) := \left\{ \rho \, \big| \, (\log b)^{-\xi_{\star}} < \rho < (\log b)^{-\xi^{\star}} \right\}.
\end{equation}
Prior to proving the theorem, we derive a uniform upper bound on the probability that $L_{\psi}(\cdot)$ attains small values within a specified range of the argument.

\begin{lem} \label{l:liminf_low}
Let $\psi(\cdot)$ be a function of the form \eqref{cond:reg} with $\varkappa \in [0,2)$, and let $\theta(\cdot)$ admit a Karamata representation \eqref{e:Kara} for $u \in (0,1)$. Suppose further that, for some $\alpha \in (0,1)$ and all sufficiently small $v > 0$, it is true that
\[
\varepsilon(v) < \frac{1}{|\log v|^{\alpha}}.
\]
Let $\xi_{\star} > 2 - \varkappa > \xi^{\star} > 0$. Then,
\begin{equation*} \label{e:L_upest}
	\limsup\limits_{b \to \infty} \sup_{\rho \in P(\xi_{\star}, \xi^{\star}, b)} \frac{\mathbb{P}\left( \frac{L_{\psi}(b)}{b^{2} \psi(b^{-1})} < \rho \right)}{\rho^{2-\varkappa}} \le e^{c_{\varkappa,1}},
\end{equation*}
where $c_{\varkappa,1} = \frac{C_{\gamma} + \log 2}{2 - \varkappa}$, and $C_{\gamma}$ denotes the Euler--Mascheroni constant.
\end{lem}

\begin{proof}
For $b > 1$ and $\rho > 0$, by using the representation \eqref{e:L_int3} at $b_{\circ} = 1$ and the definition of the measure $\mu_{\psi, b}$, we obtain
\begin{equation} \label{e:L_psi_mu1}
	\mathbb{P}\left( \frac{L_{\psi}(b)}{b^{2} \psi(b^{-1})} < \rho \right) \le 
	\exp\left(- \mu_{\psi, b}([\rho, \infty)) \right) = \exp(-\mu_1(Q_{\psi}(b, \rho))),
\end{equation}
where
$
Q_{\psi}(b, \rho) := \left\{ (s,z) \, \big| \, s \in (b^{-1}, 1], \ z > 0, \ z^{2} s^{2-\varkappa} \frac{\theta((sb)^{-1})}{\theta(b^{-1})} \ge \rho \right\}.
$ \\
Fix $\delta > 0$. For $z, \rho > 0$, define the function
\begin{equation} \label{def:s1}
      s_1(\rho, z) := ((1+\delta) \rho z^{-2})^{1 / (2 - \varkappa)}.
\end{equation}
For $b > 1$ and $\rho > 0$, define the set
\[
\widetilde{Q}_{\psi}(b, \rho) := \left\{ (s,z) \, \big| \, \sqrt{(1+\delta) \rho} \le z \le \log b, \quad s_1(\rho,z) \le s \le 1 \right\}.
\]
For sufficiently large $b$, and for any $\rho \in P(\xi_{\star}, \xi^{\star}, b)$ and $(s,z) \in \widetilde{Q}_{\psi}(b, \rho)$, it holds that
\begin{equation} \label{e:s1}
       1 \ge s \ge s_1(\rho,z) \ge (\log b)^{-(2+\xi_{\star}) / (2 - \varkappa)} > b^{-1},
\end{equation}
and also
\begin{eqnarray} \nonumber
	\frac{\theta((sb)^{-1})}{\theta(b^{-1})} &=& \frac{c((sb)^{-1})}{c(b^{-1})}\exp \left(- \int_{b^{-1}}^{(sb)^{-1}} \frac{\varepsilon(v)}{v} \mathrm{d}v   \right)   
\\ 
     &\ge& \nonumber
    \left(1+\frac{\delta}{2}\right)^{-1} \exp \left(  \frac{-(\log b)^{1-\alpha} +  (\log(sb))^{1-\alpha}}{1-\alpha}\right)   
\\ 
     &\ge& \nonumber
     \left(1+\frac{\delta}{2}\right)^{-1}\exp \left( \frac{\log s_1(b,\rho, z)}{(1-\alpha)
     (\log b)^\alpha}\right)   
\\ 
     &\ge& \label{e:theta_theta_est}
     \left(1+\frac{\delta}{2}\right)^{-1}\exp \left( \frac{-(2+\xi_{\star})\log\log b}{(1-\alpha)(2-\varkappa)
     (\log b)^{\alpha}} \right) \ge (1+\delta)^{-1}. 
\end{eqnarray}
It follows that for sufficiently large $b$, and for any $\rho \in P(\xi_{\star}, \xi^{\star}, b)$,
\[
\widetilde{Q}_{\psi}(b, \rho) \subset Q_{\psi}(b, \rho).
\]
Thus, as $b \to \infty$, for any $\rho \in P(\xi_{\star}, \xi^{\star}, b)$, we obtain
\begin{eqnarray} \nonumber
   &&	\mu_{\psi ,b}([\rho, \infty))  = \mu_1(Q_{\psi }(b,\rho)) \ge \mu_{1}(\widetilde{Q} _{\psi }(b,\rho))
	\begin{comment}
	\int_{0}^{\infty} \int_{b^{-1}}^{1} 1_{\left\{ z^2s^{2-\varkappa}\frac{\theta((sb)^{-1})}{\theta(b^{-1})} \ge r(b) \right\} }   \mu_1(\mathrm{d}s \times \mathrm{d}z)
	\end{comment}
\\ 
    &\ge&  \nonumber
\int_{\sqrt{(1+\delta)\rho} }^{\log b} \int_{s_1(\rho,z)}^{1} 2 \varphi (z) s ^{-1}  \mathrm{d}s \mathrm{d}z
\\ 
   &=& \nonumber
   \int_{\sqrt{(1+\delta)\rho}}^{\log b} \frac{2}{2-\varkappa}|\log\left((1+\delta)\rho z^{-2}\right)| \, \varphi (z) \mathrm{d}z
\\
    &=& \nonumber
    \frac{2}{2-\varkappa}|\log \rho|\int_{\sqrt{(1+\delta)\rho}}^{\log b} \varphi (z) \mathrm{d}z +
    \int_{\sqrt{(1+\delta)\rho}}^{\log b} \frac{2}{2-\varkappa}\log\left(z^{2}\right) \, \varphi (z) \mathrm{d}z - \delta_1+ o(1)
\\
    &=& \nonumber
    \frac{1}{2-\varkappa} |\log \rho|\left(1 + O\left((\log b)^{-\xi^{\star} / 2} \right)\right) + \frac{(1+o(1)) }{\sqrt{2\pi} (2-\varkappa)}\int_{0 }^{\infty} \zeta^{-1 / 2} e^{-\zeta / 2} \log \zeta  \, \mathrm{d} \zeta - \delta_1+ o(1)
\\ 
    &=& \label{e:mu_b_est_low}
    \frac{1}{2-\varkappa} |\log \rho| - \frac{C_{\gamma} + \log 2 }{2-\varkappa}-\delta_1+ o(1).
\end{eqnarray}
Here $\delta_1 := \frac{2 \log(1 + \delta)}{2 - \varkappa} > 0$  can be made arbitrarily small by an appropriate choice of $\delta$; the integral is computed in \cite[4.352(1)]{GR63}. Substituting this estimate into \eqref{e:L_psi_mu1}, we obtain
\[
\limsup\limits_{b \to \infty} \sup_{\rho \in P(\xi_{\star}, \xi^{\star}, b)} \frac{\mathbb{P}\left( \frac{L_{\psi}(b)}{b^{2} \psi(b^{-1})} < \rho \right)}{\rho^{2 - \varkappa}} \le e^{c_{\varkappa, 1} + \delta_1}.
\]
Taking the limit $\delta_1 \to 0$ yields the desired inequality.
\end{proof}

\begin{rem}
In contrast to Theorem \ref{t:liminf_low}, Lemma \ref{l:liminf_low} permits $\varkappa \in [0,1)$.
\end{rem}

The proof of Theorem \ref{t:liminf_low} is based on the idea that the probability of both the value of the function $L_{\psi}(\cdot)$ at the time of the last jump $\tau(\cdot)$ preceding a fixed time $T$, and the moment of this last jump itself being small, is low.

 \subsection*{Proof of Theorem \ref{t:liminf_low}}
For $\rho > 0$ and $b,d \ge 1$, consider the events
\begin{eqnarray*}
	A^{(\tau)}(b,d) &:=& \left\{ N\left((b, eb] \times (b^2d, \infty)\right) \ge 1  \right\}, 
\\
	A^{(L)}(b, \rho) &:=& \left\{ L_{\psi}(b) \le b^{2}\psi(b^{-1}) \rho \right\}, 
\\ 
	A(b,\rho,d) &:=& A^{(\tau)}(b,d) \cap A^{(L)}(b,\rho).
\end{eqnarray*} 
The occurrence of the event $A^{(\tau)}(b,d)$ means that the process $\tau(\cdot)$ on the interval $(b, eb]$ has jumps of size greater than $b^2 d$. 
We estimate the probability of this event using the fact that, for $\psi(\cdot) \equiv 1$ and $a \ge a_{\circ} \ge 1$, it is true that $L_{\psi}(a_{\circ}, a) = \tau(a) - \tau(a_{\circ})$. Using representation \eqref{e:L_int2}, we obtain  
\begin{eqnarray*}
     \mathbb{P} (A^{(\tau)}(b,d)) &=& 1 - \exp \left( -  \int_{0}^{\infty} \int_{1}^{e}1_{\left\{ z^2s ^2 > d \right\} }\, \mu_{1}(\mathrm{d}s\times \mathrm{d}z) \right)  
\\ 
    &\le&
	\int_{1}^{e} \int_{\frac{\sqrt{d} }{s} }^{\infty} 2s ^{-1} \varphi (z) \, \mathrm{d}z\,\mathrm{d}s
\\ 
    &\le&
	\sqrt{\frac{2}{\pi}}  \int_{1}^{e} \exp \left( -\frac{d}{2s^2} \right) \frac{1}{\sqrt{d} } \, \mathrm{d}s  
\\ 
    &\le&
	\sqrt{\frac{2}{\pi}} e^{-d / (2e^2) + 1}.
\end{eqnarray*} 
The events $A^{(\tau)}(b,d)$ and $A^{(L)}(b,\rho)$ are independent because the random variable $L_{\psi}(b) = \int\limits_1^{b} \psi(a^{-1}) \, \mathrm{d}\tau(a)$ does not depend on the process $\left\{ \tau(a) - \tau(b) \, | \, a \in [b, eb] \right\}$, due to the independent increments property of $\tau(\cdot)$.
Using Lemma \ref{l:liminf_low}, for sufficiently large $b$ and for any $\rho \in P(\xi_{\star}, \xi^{\star}, b)$ with some $\xi_{\star} > 2- \varkappa > \xi^{\star} > 0$, we obtain
\[
        \mathbb{P} (A^{(L)}(b,\rho)) \le 2 e^{c_{\varkappa,1}} \rho^{1 / (2-\varkappa)}.
\] 
Hence, for sufficiently large $b$ and for any $d \ge 1$ and $\rho \in P(\xi_{\star}, \xi^{\star}, b)$, we have 
\begin{equation} \label{e:A}
       \mathbb{P} (A(b,\rho,d)) \le c_{\varkappa,2} \, \rho^{1 / (2-\varkappa)} e^{-d / (2e^2)}, 
\end{equation} 
where $c_{\varkappa,2} := 2 e^{c_{\varkappa,1}+1} \sqrt{\frac{2}{\pi}}$.
For integers $n \ge 1$ and $0 \le m \le m_{*} := \left\lfloor 5 + \log \log n \right\rfloor$, consider the sequences
\begin{eqnarray*}
	b_{n} &:=& e^{n}, 
\\
	\rho_{n,m} &:=& g(b_{n}^2) e^{(2-\varkappa)(m+1)}, 
\\
	d_{m} &:=& e^{m} 
\end{eqnarray*}
and the events
\[
    A_{n,m} := A(b_{n}, \rho_{n,m}, d_{m}), \quad U_{n} := A^{(L)}(b_{n},\rho_{n,0}) \cup \left( \bigcup_{m = 1}^{m_{*}} A_{n,m} \right).
\]
By using condition \eqref{cond:t1g1}, for sufficiently large $n$ and for any $m \in [0, m_{*}]$, we obtain
\[
	(\log b_{n})^{-\xi_{\circ}} \le \rho_{n,m} \le (\log b_{n})^{-(\xi^{\circ} / 2)},
\] 
that is, $\rho_{n,m} \in P(\xi_{\circ}, \frac{\xi^{\circ}}{2} , b_{n})$.
Rewriting \eqref{e:A} in the new notation, we obtain
\[
    \mathbb{P} (A_{n,m}) \le c_{\varkappa,2} \, g(b_{n}^2)^{1 / (2-\varkappa)} \exp\left(-e^{m} / (2e^2)\right) e^{m+1}.
\] 
From this estimate, it follows that for sufficiently large $n$, we have 
\begin{eqnarray*}
    \mathbb{P} (U_{n}) \le \mathbb{P}\left(A^{(L)}(b_{n},\rho_{n,0})\right)  + \sum_{m=1}^{\infty}   \mathbb{P} (A_{n,m}) \le c_{\varkappa,3} \, g(b_{n}^2)^{1 / (2-\varkappa)}, 
\end{eqnarray*}
for some constant $c_{\varkappa,3} > 0$ independent of $n$. Therefore, by using condition \eqref{cond:t1g2} and the fact that $g(\cdot)$ is a slowly varying function, we obtain
\[
       \sum_{n=1}^{\infty} \mathbb{P} \left( U_{n} \right)  < \infty.   
\] 
Consequently, by the Borel--Cantelli lemma, almost surely, for sufficiently large $n$, the events $U_{n}$ do not occur. 

For $T > 0$, take $n = n(T)$ such that
\[ 
      \tau(b_{n}) \le T < \tau(b_{n+1}).
\]
The non-occurrence of $U_{n}$ implies that, for every integer $m \in [1, m_{*}]$, at least one of the following events occurs: 
\begin{eqnarray*} 
	\label{cond:L}
	\overline{A_{n,m}^{(L)}} &:=&  
	\left\{ L_{\psi}(b_{n}) > b_{n}^{2}\psi(b_{n}^{-1}) \rho_{n,m} \right\}, 
\\    \label{cond:tau}
	\overline{A_{n,m}^{(\tau)}} &:=&  \left\{ N\left((b_{n}, b_{n+1}] \times (b_{n}^2 d_{m}, \infty)\right) = 0 \right\},
\end{eqnarray*} 
It also implies that the event $\overline{A_{n,0}^{(L)}}$ occurs.
\\
By Lemma \ref{l:tau_est}, for sufficiently large $T$, the following inequalities hold:
\begin{equation} \label{e:b_tau}
	 b_{n} \ge \frac{1}{4} T^{1 / 2} (\log \log T)^{-1/2}
\end{equation} 
and
\begin{equation} \label{e:b_tau2}
        \tau(b_{n+1}) < e b_{n+1}^2 \log \log b_{n+1}.
\end{equation} 

By using \eqref{e:b_tau} and condition \eqref{cond:eps_alpha}, similarly to inequality \eqref{e:theta_theta_est}, for sufficiently large $T$ such that $b_{n} < T^{1/2}$, we obtain
\begin{equation} \label{e:th_th_est}
	\frac{\theta(T^{-1/2})}{\theta(b_{n}^{-1})} 
	 \le 2.
\end{equation}

From inequality \eqref{e:b_tau2}, it follows that for sufficiently large $T$, the event $\overline{A_{n,m_{*}}^{(\tau)}}$ occurs.
For $T > 0$, consider the smallest $m = m(T) \in [1,m_{*}]$ for which $\overline{A_{n,m}^{(\tau)}}$ occurs; that is, the process $\tau(\cdot)$ on $(b_{n}, b_{n+1}]$ does not have jumps larger than $b_{n}^2 d_{m}$. Exactly one of the following three cases can occur:
\begin{eqnarray}
	b_{n}^2 d_{m} &<& \frac{T}{4},\quad b_{n} < T^{1/2}, \label{e:case1}
	\\
	b_{n}^2 d_{m} &\ge& \frac{T}{4},\quad b_{n} < T^{1/2}, \label{e:case2}
	\\
	b_{n}^2 d_{m} &\ge& \frac{T}{4},\quad b_{n} \ge T^{1/2}. \label{e:case3}
\end{eqnarray}

In case \eqref{e:case1}, since the process $\tau(\cdot)$ on the interval $(b_n,b_{n+1}]$ has no jumps exceeding $T/4$, there exists some $b \in [b_n,b_{n+1}]$ such that $\tau(b) \in [\frac{3T}{4}, T]$. From this, by using Proposition \ref{p:L_S}, together with \eqref{e:th_th_est} and \eqref{e:case1}, it follows that for sufficiently large $T$ corresponding to this case, we have
\begin{eqnarray} \nonumber
	 \frac{L_{\psi}(b)}{T^{1-\varkappa / 2}\theta(T^{-1/2}) g(T)} 
      &\ge& 	
    	\frac{\tau(b) \psi(b^{-1})}{T^{1-\varkappa / 2}\theta(T^{-1/2}) g(T)} 
    \\ &\ge& \nonumber
\frac{3}{8} \cdot  
     \frac{b_{n+1}^{-\varkappa} \theta(b_{n}^{-1})}{T^{-\varkappa / 2}\theta(T^{-1/2}) g(T)}
     \\  &\ge& \label{e:liminf0}
     \frac{3}{16} \left(\frac{e}{2}\right)^{-\varkappa}	
     \frac{1}{g(T)}.
\end{eqnarray}

In case \eqref{e:case2}, it is true that
\begin{equation} \label{e:b_ineq}
	b_{n} \ge \frac{1}{2} T^{1/2} e^{-m/2}.
\end{equation}
For sufficiently large $T$ corresponding to this case, by using \eqref{e:th_th_est}, \eqref{e:b_ineq} and the occurrence of $\overline{A_{n,m-1}^{(L)}}$, we obtain
\begin{eqnarray} \nonumber 
 \frac{L_{\psi}(b_{n})}{T^{1-\varkappa / 2} \theta(T^{-1/2}) g(T)} 
	 &\ge& 
	\frac{b_{n}^{2- 
        \varkappa} \theta(b_{n}^{-1}) \rho_{n,m-1}}{T^{1-\varkappa / 2} \theta(T^{-1/2}) g(T)}
\\ 
   &\ge& \label{e:liminf1} 
        \frac{2^{\varkappa - 2}}{2} \cdot \frac{e^{m(2-\varkappa)/2} g(b_{n}^2)}{g(T)} 
   \ge \frac{2^{\varkappa-2}}{4}.
\end{eqnarray}
In the last step, we used \eqref{e:b_ineq} and the fact that, beyond some point, $g(\cdot)$ grows more slowly than $x^{(2-\varkappa)/2}$.

In case \eqref{e:case3}, for corresponding sufficiently large $T$, by using the occurrence of $\overline{A_{n,0}^{(L)}}$, we obtain
\begin{eqnarray} \label{e:liminf2}
	\frac{L_{\psi}(b_{n})}{T^{1-\varkappa / 2} \theta(T^{-1/2}) g(T)} &\ge&
	\frac{e^{2-\varkappa} \, b_{n}^{2-\varkappa} \theta(b_{n}^{-1}) g(b_{n}^2)}{T^{1-\varkappa / 2} \theta(T^{-1/2}) g(T)} \ge \frac{e^{2-\varkappa}}{2}.
\end{eqnarray}
In the last inequality, we used the fact that the function $b \mapsto \theta(b^{-1/2}) g(b)$ is slowly varying at infinity.
Due to the monotonicity of $I_W^{\psi}(\cdot, r)$, we have
\begin{equation} \label{e:iw}
I_W^{\psi}(T,r) \ge I_W^{\psi}(\tau(\widetilde{b}), r), \quad T > 0,
\end{equation}
where $\widetilde{b} = b$ in case \eqref{e:case1} and $\widetilde{b} = b_{n}$ in the other two cases.
Combining inequalities \eqref{e:liminf0}, \eqref{e:liminf1}, and \eqref{e:liminf2}, and applying Proposition \ref{p:L_I} and \eqref{e:iw}, we obtain, with probability 1, that
\[
\liminf_{T \to \infty} \frac{I_W^{\psi}(T,r)}{T^{1-\varkappa / 2} \theta(T^{-1/2}) g(T)} \ge 2^{\varkappa - 4}.
\]
This inequality remains true if $g(\cdot)$ is replaced by $p g(\cdot)$ for any $p > 0$, thus yielding the assertion of the theorem.

\subsection{Proof of Theorem \ref{t:liminf_up}}

We begin the proof by establishing the upper bounds on the tails of the measure $\mu_{\psi,b}$.  
\begin{lem} \label{l:mu_upper}
Let $\psi(\cdot)$ be a function of the form \eqref{cond:reg} with $\varkappa \in [0,2)$, and suppose $\theta(\cdot)$ admits the Karamata representation \eqref{e:Kara} for $u \in (0, 1)$. Assume further that there exists $\alpha \in (0,1)$ such that, for all sufficiently small $v > 0$,
\[
|\varepsilon(v)| < \frac{1}{|\log v|^{\alpha}}.
\]
Let $\xi_{\star} > 2 - \varkappa > \xi^{\star} > 0$.
Then we have
\begin{eqnarray} \label{e:mu_local}
\limsup_{b \to \infty} \sup_{\ell \in P(\xi_{\star}, \xi^{\star}, b)} \left( \mu_{\psi,b}([\ell,\infty)) - \frac{1}{2-\varkappa}|\log \ell| \right) \le - c_{\varkappa,1}.
\end{eqnarray}
Here $P(\xi_{\star}, \xi^{\star}, b)$ is defined by \eqref{def:P}, and $c_{\varkappa,1}$ is defined in the statement of Lemma \ref{l:liminf_low}.
Moreover, for sufficiently large $b$ and for all $\ell \in (0,1)$, it holds that
\begin{eqnarray} \label{e:uniform}
\mu_{\psi,b}([\ell,\infty)) \le \frac{1}{2-\varkappa}|\log \ell| + \frac{2}{2-\varkappa}\frac{(\log b)^{1-\alpha}}{1-\alpha} + c_{\varkappa,4},
\end{eqnarray}
for some $c_{\varkappa,4} \in \mathbb{R}$.
\end{lem}
\begin{proof}
Fix $\delta \in (0,1)$.
There exists $\beta > 1$ such that, for all $b > \beta$ and $s \in [\beta b^{-1}, 1]$, we have
\begin{eqnarray} \nonumber
\frac{\theta((sb)^{-1})}{\theta(b^{-1})} &=& \frac{c((sb)^{-1})}{c(b^{-1})} \exp\left(- \int_{b^{-1}}^{(sb)^{-1}} \frac{\varepsilon(v)}{v} \mathrm{d}v \right) \\
&\le& \nonumber
(1-\delta)^{-1} \exp\left( \frac{1}{1-\alpha} \big( (\log b)^{1-\alpha} - (\log(sb))^{1-\alpha} \big) \right) \\
&\le& \label{e:theta_theta_est4}
(1-\delta)^{-1} \exp\left( \frac{|\log s|}{(1-\alpha)(\log b)^\alpha} \right) \\
&\le& \label{e:theta_theta_est2}
(1-\delta)^{-1} \exp\left( \frac{(\log b)^{1-\alpha}}{1-\alpha} \right).
\end{eqnarray}
Define, for $\ell, z > 0$ and $b > 1$,
\[
s_2(\ell,b,z) := \left( (1-\delta) \ell z^{-2} \exp\left(- \frac{(\log b)^{1-\alpha}}{1-\alpha} \right) \right)^{1 / (2-\varkappa)}.
\]
For $b > \beta$, and $\ell,z > 0$, and $\beta b^{-1} \le s < s_2(\ell,b,z)$, applying \eqref{e:theta_theta_est2} yields
\[
z^2 s^{2-\varkappa} \frac{\theta((sb)^{-1})}{\theta(b^{-1})} < \ell.
\]
For $\ell > 0$, and $z \ge \sqrt{(1-\delta)\ell}$, and $b > 1$, it is true that
\[
s_2(\ell,b,z) \le 1.
\]
Thus, for $\ell \in (0,1)$ and $b > \beta$, we obtain
\begin{eqnarray*} \nonumber
\mu_{\psi,b}([\ell,\infty)) &=& \int_0^\infty \int_{b^{-1}}^1 1_{\left\{ z^2 s^{2-\varkappa} \frac{\theta((sb)^{-1})}{\theta(b^{-1})} \ge \ell \right\}} \mu_1(\mathrm{d}s \times \mathrm{d}z) \\
&\le& \nonumber
\int_{\sqrt{(1-\delta)\ell}}^{\infty} \int_{s_2(\ell,b,z)}^1 2 \varphi(z) s^{-1} \mathrm{d}s \mathrm{d}z + \int_0^\infty \int_{b^{-1}}^{\beta b^{-1}} 2 \varphi(z) s^{-1} \mathrm{d}s \mathrm{d}z \\
&=& \label{e:mu_b_est_inter}
\int_{\sqrt{(1-\delta)\ell}}^{\infty} \frac{2}{2-\varkappa} \left| \log \left( (1-\delta)\ell z^{-2} \right) - \frac{(\log b)^{1-\alpha}}{1-\alpha} \right| \varphi(z) \mathrm{d}z + \log \beta \\
&\le& \label{e:mu_b_est_up}
\frac{1}{2-\varkappa} |\log \ell| + \frac{2}{2-\varkappa} \frac{(\log b)^{1-\alpha}}{1-\alpha} + \frac{\int_1^\infty \zeta^{-1/2} e^{-\zeta/2} \log \zeta \, \mathrm{d}\zeta}{\sqrt{2\pi} (2-\varkappa)} + \delta_2 + \log \beta,
\end{eqnarray*}
where $\delta_2 := \frac{2|\log(1-\delta)|}{2-\varkappa} > 0$. Hence, inequality \eqref{e:uniform} holds with
\[
c_{\varkappa,4} := \frac{\int_1^\infty \zeta^{-1/2} e^{-\zeta/2} \log \zeta \, \mathrm{d}\zeta}{\sqrt{2\pi} (2-\varkappa)} + \delta_2 + \log \beta.
\]

Fix any $\gamma \in (0,\alpha)$. Similarly to \eqref{e:theta_theta_est2}, for sufficiently large $b$ and \\ 
$ s \in [\exp(-(\log b)^\gamma), 1]$, we have
\begin{equation} \label{e:theta_theta_est3}
\frac{\theta((sb)^{-1})}{\theta(b^{-1})} \le \left(1 - \frac{\delta}{2}\right)^{-1} \exp\left( \frac{|\log s|}{(1-\alpha)(\log b)^\alpha} \right) \le (1-\delta)^{-1}.
\end{equation}
Define, for $\ell, z > 0$, 
\[
s_3(\ell,z) := \left( (1-\delta) \ell z^{-2} \right)^{1/ (2- \varkappa)}.
\]
For $\ell > 0$ and $z \ge \sqrt{(1-\delta) \ell}$, it is true that
\[
s_3(\ell,z) \le 1.
\]
For sufficiently large $b$ and for any $\ell,z > 0$ and $s_3(\ell,z) > s \ge \exp(-(\log b)^\gamma)$, by using \eqref{e:theta_theta_est3}, we get
\[
z^2 s^{2-\varkappa} \frac{\theta((sb)^{-1})}{\theta(b^{-1})} < \ell.
\]
For sufficiently large $b$ and $\beta b^{-1} < s < \exp(-(\log b)^\gamma)$, $z \in (0, \log b)$, by using \eqref{e:theta_theta_est4}, we obtain
\begin{eqnarray*}
z^2 s^{2-\varkappa} \frac{\theta((sb)^{-1})}{\theta(b^{-1})} &<& (1-\delta)^{-1} (\log b)^2 \exp\left( \frac{|\log s|}{(1-\alpha)(\log b)^\alpha} \right) s^{2-\varkappa} \\
&=& (1-\delta)^{-1} (\log b)^2 \exp\left( \left( 2-\varkappa - \frac{1}{(1-\alpha)(\log b)^\alpha} \right) \log s \right) \\
&<& (\log b)^{-\xi_{\star}}.
\end{eqnarray*}
Since $c(\cdot)$ and $\varepsilon(\cdot)$ are bounded functions, for sufficiently large $b$ and $z \in (0, \log b)$, $s \in (b^{-1}, \beta b^{-1})$, we have
\[
	z^2 s^{2-\varkappa} \frac{\theta((sb)^{-1})}{\theta(b^{-1})} = z^2 s^2 b^{\varkappa} \frac{\psi((sb)^{-1})}{\theta(b^{-1})} \le \beta^2 \frac{b^{-2+\varkappa} (\log b)^2}{\theta(b^{-1})} \sup_{u \in [1,\beta]} \psi(u) < (\log b)^{-\xi_{\star}}.
\]
Thus, as $b \to \infty$, for any $\ell \in P(\xi_{\star}, \xi^{\star}, b)$, it holds that
\begin{eqnarray*} \label{e:mu_b_est_inter}
	\mu_{\psi,b}([\ell, \infty)) &=& \int_0^\infty \int_{b^{-1}}^1 1_{\left\{ z^2 s^{2-\varkappa} \frac{\theta((sb)^{-1})}{\theta(b^{-1})} \ge \ell \right\}} \mu_1(\mathrm{d}s \times \mathrm{d}z) \\
	&\le& \int_{\sqrt{(1-\delta)\ell}}^{\log b} \int_{s_3(\ell,z)}^{1} 2 \varphi(z) s^{-1} \mathrm{d}s \mathrm{d}z + \int_{\log b}^\infty \log b \, \varphi(z) \mathrm{d}z \\
	&=& \int_{\sqrt{(1-\delta)\ell}}^\infty \frac{2}{2-\varkappa} \big|\log\big((1-\delta)\ell z^{-2}\big)\big| \, \varphi(z) \mathrm{d}z + o(1) \\
	&=& \frac{1}{2-\varkappa} |\log \ell| - c_{\varkappa,1} + \delta_2 + o(1).
\end{eqnarray*}
The last equality is verified similarly to \eqref{e:mu_b_est_low}.
Inequality \eqref{e:mu_local} now follows by letting $\delta_2 \to 0$ through an appropriate choice of $\delta$.
\end{proof}

Next, we establish a sharp estimate for the probabilities of  small values of $L_{\psi}(\cdot,\cdot)$, uniform in the first argument when it is sufficiently small.

\begin{lem} \label{l:liminf_up}
Let $\psi(\cdot)$ be a function of the form \eqref{cond:reg} with $\varkappa \in [0,2)$, and suppose $\theta(\cdot)$ admits the Karamata representation \eqref{e:Kara} for $u \in (0, 1)$. Assume further that there exists $\alpha \in (0,1)$ such that, for all sufficiently small $v > 0$,
\[
|\varepsilon(v)| < \frac{1}{|\log v|^{\alpha}}.
\]
Let $\xi_{\circ} > 2 - \varkappa > \xi^{\circ} > 0$ and let $\rho(\cdot)$ be a function such that, for sufficiently large $b$, it holds that
\[
	\rho(b) \in P(\xi_{\circ}, \xi^{\circ}, b).
\]
Let $\widehat{b} := b (\log b)^{-(4+\xi_{\circ})/(2-\varkappa)}$. Then, for some constant $c_{\varkappa,5} > 0$, we have
\begin{equation*}
	\lim_{b \to \infty} \sup_{b_{\circ} \in [1, \widehat{b}]} \frac{\mathbb{P}\left(\frac{L_{\psi}(b_{\circ}, b)}{b^{2} \psi(b^{-1})} < \rho(b) \right)}{\rho(b)^{1/(2-\varkappa)}} = 
	\lim_{b \to \infty} \inf_{b_{\circ} \in [1, \widehat{b}]} \frac{\mathbb{P}\left(\frac{L_{\psi}(b_{\circ}, b)}{b^{2} \psi(b^{-1})} < \rho(b) \right)}{\rho(b)^{1/(2-\varkappa)}} = c_{\varkappa,5}.
\end{equation*}
\end{lem}

\begin{proof}
For $b \ge b_{\circ} \ge 1$, consider the events
\begin{eqnarray*}
	A_{b_{\circ},b}^{\dag} &:=& \left\{ \int_{(0, \rho(b))} \ell \, N_{\psi,b_{\circ},b}(\mathrm{d}\ell) < \rho(b) \right\}, \\
	A_{b_{\circ},b}^{\ddag} &:=& \left\{ N_{\psi,b_{\circ},b}([\rho(b), \infty)) = 0 \right\}.
\end{eqnarray*}
Using the independence of $A_{b_{\circ},b}^{\dag}$ and $A_{b_{\circ},b}^{\ddag}$ and representation \eqref{e:L_int3}, we get
\begin{equation} \label{e:Lpsi1}
	\mathbb{P}\left(\frac{L_{\psi}(b_{\circ},b)}{b^{2} \psi(b^{-1})} < \rho(b) \right) = \mathbb{P}(A_{b_{\circ},b}^{\dag}) \cdot \mathbb{P}(A_{b_{\circ},b}^{\ddag}) = \mathbb{P}(A_{b_{\circ},b}^{\dag}) \cdot \exp\left(-\mu_{\psi,b_{\circ},b}([\rho(b), \infty))\right).
\end{equation}

Define the measure
\[
\widetilde{\mu}_{\psi,b_{\circ},b}([\widetilde{\ell}, 1]) := \mu_{\psi,b_{\circ},b}([\widetilde{\ell} \, \rho(b), \rho(b)]), \quad \widetilde{\ell} \in (0,1].
\]
Let $\widetilde{N}_{\psi,b_{\circ},b}(\mathrm{d}\widetilde{\ell})$ be a Poisson random measure with intensity measure $\widetilde{\mu}_{\psi,b_{\circ},b}(\mathrm{d}\widetilde{\ell})$, then
\[
\int_{(0, \rho(b))} \ell \, N_{\psi,b_{\circ},b}(\mathrm{d}\ell) = \rho(b) \int_{(0,1)} \widetilde{\ell} \, \widetilde{N}_{\psi,b_{\circ},b}(\mathrm{d}\widetilde{\ell}).
\]

For any $\ell > 0$ and $b \ge b_{\circ} \ge 1$, since $L_{\psi}(b_{\circ}, b) \le L_{\psi}(b)$. it is true that
\begin{equation} \label{e:adj}
	\mu_{\psi,b_{\circ},b}([\ell, \infty)) \le \mu_{\psi,b}([\ell, \infty)),
\end{equation}
Consequently, the inequalities established in Lemma \ref{l:mu_upper} with $\xi_{\star} = \xi_{\circ} + 2$ and $\xi^{\star} = \xi^{\circ}$ apply to $\mu_{\psi,b_{\circ},b}$. Using inequality \eqref{e:mu_local}, for any $\delta_1 > 0$, as $b \to \infty$, for any $b_{\circ} \in [1,b]$ and $\ell \in P(\xi_{\star}, \xi^{\star}, b)$, it holds that
\begin{equation} \label{e:mu_circ_upper}
	\mu_{\psi,b_{\circ},b}([\ell, \infty)) \le \frac{1}{2-\varkappa} |\log \ell| - c_{\varkappa,1} + \delta_1.
\end{equation}
For all sufficiently large $b$, and for any $b_{\circ} \in [1, \widehat{b}]$, $\ell \in P(\xi_{\star}, \xi^{\star}, b)$, and $\sqrt{(1+\delta) \ell} \le z \le \log b$, the function $s_1(\cdot, \cdot)$ defined in \eqref{def:s1} satisfies
\[
1 \ge s_1(\ell, z) \ge (\log b)^{-(2+\xi_{\star})/(2-\varkappa)} = \frac{\widehat{b}}{b} \ge \frac{b_{\circ}}{b}.
\]
Hence, by repeating the calculation in \eqref{e:mu_b_est_low} verbatim, we obtain that for any $\delta_2 > 0$, and for all sufficiently large $b$, uniformly over all $b_{\circ} \in [1, \widehat{b}]$ and $\ell \in P(\xi_{\star}, \xi^{\star}, b)$, the following holds:
\begin{equation} \label{e:mu_circ_lower}
	\mu_{\psi,b_{\circ},b}([\ell, \infty)) \ge \frac{1}{2-\varkappa} |\log \ell| - c_{\varkappa,1} - \delta_2.
\end{equation}
By combining estimates \eqref{e:mu_circ_upper} and \eqref{e:mu_circ_lower}, we obtain that for any $\delta_3 > 0$, and for all sufficiently large $b$, uniformly over all $b_{\circ} \in [1, \widehat{b}]$ and $\widetilde{\ell} \in [(\log b)^{-2}, 1]$, it holds that
\begin{eqnarray} \label{e:mu_tilde_upper}
	\widetilde{\mu}_{\psi,b_{\circ},b}([\widetilde{\ell},1]) &=& \mu_{\psi,b_{\circ},b}([\widetilde{\ell} \rho(b), \infty)) - \mu_{\psi,b_{\circ},b}((\rho(b), \infty)) \le \frac{1}{2-\varkappa} |\log \widetilde{\ell}| + \delta_3, \\
	\label{e:mu_tilde_lower}
	\widetilde{\mu}_{\psi,b_{\circ},b}([\widetilde{\ell},1]) &\ge& \frac{1}{2-\varkappa} |\log \widetilde{\ell}| - \delta_3.
\end{eqnarray}
Thus, for any fixed $\widetilde{\ell} \in (0,1]$, as $b \to \infty$, the following holds uniformly over all $b_{\circ} \in [1, \widehat{b}]$:
\begin{equation} \label{e:mu_to_mu}
	\widetilde{\mu}_{\psi,b_{\circ},b}([\widetilde{\ell},1]) \to \widetilde{\mu}_{\varkappa}([\widetilde{\ell},1]) := \frac{1}{2-\varkappa} |\log \widetilde{\ell}|.
\end{equation}

By combining estimates \eqref{e:mu_tilde_upper} and \eqref{e:mu_tilde_lower}, together with \eqref{e:uniform}, \eqref{e:adj}, and \eqref{e:mu_circ_lower}, for any $\delta_4, \epsilon >0$, as $b \to \infty$, uniformly over all $b_{\circ} \in [1, \widehat{b}\,]$, we obtain
\begin{eqnarray*} \nonumber
	\int_{0}^{\epsilon} \widetilde{\ell}  \widetilde{\mu} _{\psi, b_{\circ},b}(\mathrm{d}\widetilde{\ell} ) &=& \int_{0}^{\epsilon} \widetilde{\mu} _{\psi, b_{\circ},b}([\widetilde{\ell} , \epsilon))\mathrm{d}\widetilde{\ell}  
\\ &\le& \nonumber
\int_{0}^{\epsilon} \left( \frac{1}{2-\varkappa}\Big|\log \frac{\widetilde{\ell} }{\epsilon}\Big| + \delta_4 \right) \mathrm{d}\widetilde{\ell}
\\ & &+ \int_{0}^{(\log b)^{-2}} \left( \frac{2}{2-\varkappa}\frac{(\log b)^{1-\alpha}}{1-\alpha} + c_{\varkappa,4} + c_{\varkappa,1} + \delta_4\right)  \mathrm{d}\widetilde{\ell}
\\ &=& \nonumber
 \left( \frac{1}{2-\varkappa} + \delta_4 \right)\epsilon + O\left((\log b)^{-1-\alpha}\right). 
\end{eqnarray*} 
It follows that
\begin{equation} \label{e:lim}
\lim_{\epsilon \to 0} \limsup\limits_{b \to \infty} \sup_{b_{\circ} \in [1, \widehat{b}]} \int_0^{\epsilon}\widetilde{\ell} \, \widetilde{\mu} _{\psi, b_{\circ},b}(\mathrm{d}\widetilde{\ell})  = 0. 
\end{equation} 

Let $\widetilde{N} _{\varkappa}(\mathrm{d}\widetilde{\ell} )$ denote a Poisson random measure with intensity measure $\widetilde{\mu} _{\varkappa}(\mathrm{d}\widetilde{\ell} )$. 
It is easy to verify that conditions \eqref{e:mu_to_mu} and \eqref{e:lim} imply pointwise convergence of the characteristic functions of $\int_{(0,1)} \widetilde{\ell}\widetilde{N} _{\psi ,b_{\circ},b}(\mathrm{d}\widetilde{\ell})$ to that of $\int_{(0,1)} \widetilde{\ell}\widetilde{N} _{\varkappa}(\mathrm{d}\widetilde{\ell} )$, as $b \to \infty$, uniformly over all $b_{\circ} \in [1, \widehat{b}\,]$. Since $\widetilde{\mu} _{\varkappa}(0,1] = \infty$, the random variable $\int_{(0,1)} \widetilde{\ell}\widetilde{N} _{\varkappa}(\mathrm{d}\widetilde{\ell} )$ has a continuous distribution, and thus, as $b \to \infty$, the following holds uniformly over all $b_{\circ}\in [1,\widehat{b}\,]$:
\begin{equation} \label{e:lim2}
 \mathbb{P} \left(  A_{b_{\circ},b}^{\dag}\right)  = \mathbb{P} \left( \int_{(0,1)} \widetilde{\ell}\widetilde{N} _{\psi ,b_{\circ},b}(\mathrm{d}\widetilde{\ell} ) < 1 \right) \to   \mathbb{P} \left( \int_{(0,1)} \widetilde{\ell}\widetilde{N} _{\varkappa}(\mathrm{d}\widetilde{\ell} ) < 1 \right) > 0.
\end{equation}
By combining inequality \eqref{e:Lpsi1}  with convergence \eqref{e:lim2}, the bounds given by \eqref{e:mu_circ_upper} and \eqref{e:mu_circ_lower}, and the monotonicity of $L_{\psi }(\cdot ,\cdot )$ in its first argument, we obtain
\[
\lim_{b \to \infty} \sup_{b_{\circ}\in  [1, \widehat{b}\,] }\frac{\mathbb{P} \left( \frac{L_{\psi}(b_{\circ},b)}{b^{2}\psi(b^{-1})} < \rho(b) \right)}{\rho(b)^{1 / (2-\varkappa)}} = 
\lim_{b \to \infty}\frac{\mathbb{P} \left( \frac{L_{\psi}(\widehat{b}\,,b)}{b^{2}\psi(b^{-1})} < \rho(b) \right)}{\rho(b)^{1 / (2-\varkappa)}} = e^{c_{\varkappa,1}} \, \mathbb{P} \left( \int_{(0,1)} \widetilde{\ell}N_{\varkappa}(\mathrm{d}\widetilde{\ell} ) < 1 \right), 
\] 
and also
\[
\lim_{b \to \infty} \inf_{b_{\circ}\in  [1, \widehat{b}\,] }\frac{\mathbb{P} \left( \frac{L_{\psi}(b_{\circ},b)}{b^{2}\psi(b^{-1})} < \rho(b) \right)}{\rho(b)^{1 / (2-\varkappa)}} = 
\lim_{b \to \infty}\frac{\mathbb{P} \left( \frac{L_{\psi}(b)}{b^{2}\psi(b^{-1})} < \rho(b) \right)}{\rho(b)^{1 / (2-\varkappa)}} = e^{c_{\varkappa,1}} \, \mathbb{P} \left( \int_{(0,1)} \widetilde{\ell}N_{\varkappa}(\mathrm{d}\widetilde{\ell} ) < 1 \right). 
\] 

\end{proof}

We will require a statement asserting that, with positive probability, the process $ W(\cdot)$  attains a local maximum at a jump time of $\tau(\cdot)$, and remains below this peak over a sufficiently long subsequent interval.
\begin{lem} \label{l:argmax}
There exist $\lambda > 0$ and $\bar{e} > 1$ such that
\[
\mathbb{P} \left( \left\{ \argmax_{t \in [0, 1]} \, W\big(\tau(\lambda^{-1}) + t\big) = 0 \right\}  
\cap \left\{\tau(\bar{e}\lambda^{-1}) - \tau(\lambda^{-1}) > 1 \right\}  \right) > 0. 
\] 
\end{lem}
\begin{proof}
For $\lambda^{\prime} \ge 0$, denote the random variable
\[
t^{(\lambda^{\prime})} := \argmax_{t \in [0, 2]} \left( W(t) - \lambda^{\prime} t \right)
\]
and the event
\[
	B_{0}^{(M)}(\lambda^{\prime}) := \Big\{t^{(0)} = t^{(\lambda^{\prime})} \Big\}.
\]
According to \cite[Exercise~1.14]{PM10}, we have
\[
	\limsup_{\epsilon \downarrow  0} \frac{W(t^{(0)}+\epsilon) - W(t^{(0)})}{\epsilon} \stackrel{\text{a.s.}}{=} -\infty.
\]
Since the process $\left\{ W(2-\,\backvec t) - W(\, \backvec t) \mid \ \backvec t \in [0,2] \right\}$ is a Wiener process on $[0,2]$, we obtain an analogous statement for the left neighborhood of the maximum time
\[
	\limsup_{\epsilon \downarrow  0} \frac{W(t^{(0)}-\epsilon) - W(t^{(0)})}{\epsilon} \stackrel{\text{a.s.}}{=} -\infty.
\]
It follows that, with probability 1, for sufficiently small $\lambda^{\prime} > 0$, the event $B_{0}^{(M)}(\lambda^{\prime})$ occurs, hence, there exists $\lambda > 0$ such that
\[
\mathbb{P} \left( B_0^{(M)}(\lambda) \cap \left\{ t^{(0)} \le 1 \right\}  \right) > 0.
\]

Denote the process $\widetilde{W} := \left\{ W(t^{\prime}+2) - W(2) \mid t^{\prime} \ge 0 \right\}$. This is a Wiener process on $[0,\infty)$ independent of $W(\cdot)$ on $[0,2]$. Using the definition \eqref{def:tau} of $\tau(\cdot)$, we obtain
\[
\mathbb{P} \left(  B_0^{(M)}(\lambda) \cap \left\{ \tau(\lambda^{-1}) \le 1 \right\}  \right) \ge \mathbb{P} \left(B_0^{(M)}(\lambda)\cap \left\{ t^{(0)} \le 1 \right\}\right) \cdot  \mathbb{P}\left(\sup_{t^{\prime} \ge 0 } \left( \widetilde{W}(t^{\prime}) - \lambda t^{\prime}  \right) < \lambda \right) > 0.
\]
The positivity of the latter factor follows from the fact that, by \cite[Corollary~2]{RP81}, the random variable $\sup_{t^{\prime} \ge 0 } \left( \widetilde{W}(t^{\prime}) - \lambda t^{\prime} \right)$ is exponentially distributed with parameter $2\lambda$.

Take sufficiently large $\bar{e} > 1$ such that
\[
\mathbb{P} \left( B_0^{(M)}(\lambda) \cap \left\{ \tau(\lambda^{-1}) \le 1 \right\} \cap \left\{ \tau(\bar{e} \lambda^{-1}) - \tau(\lambda^{-1}) > 1 \right\} \right) > 0.
\]
Such $\bar{e}$ exists because $\mathbb{P} \left( \left\{ \tau(e^{\prime} \lambda^{-1}) - \tau(\lambda^{-1}) > 1 \right\} \right)$ tends to 1 as $e^{\prime} \to \infty$. This concludes the proof of the lemma with the chosen values of  $\lambda$ and $\bar{e}$.

\end{proof}

 \subsection*{Proof of Theorem \ref{t:liminf_up}}
Let us take $\lambda > 0, \bar{e} > 1$ satisfying the assertion of Lemma \ref{l:argmax}. For integers $j \ge 1$, consider the sequences
\begin{eqnarray*}
	b_{j} &:=& \bar{e}^{\,j}, \\
	\rho_{j} &:=& g(b_{j}^2).
\end{eqnarray*}
For integers $j \ge i \ge 1$, consider the events
\begin{eqnarray*}
	B^{(\tau)}_{j} &:=& \left\{ \tau(b_{j+1}) - \tau(b_{j}) > (\lambda b_{j})^2 \right\}, 
		      \\ 
	B^{(M)}_{j} &:=& \Big\{ \argmax_{t \in [0, \, (\lambda b_{j})^2]}\, W(\tau(b_{j})+t) = 0 \Big\}, 
\\
	B^{(\tau,M)}_{j} &:=& B^{(\tau)}_{j} \cap B^{(M)}_{j},
\\
	B^{(L)}_{j} &:=& \left\{ L_{\psi }(b_{j}) \le b_{j}^{2}\psi(b_{j}^{-1}) \rho_{j}\right\}, \\ 
	B^{(L)}_{i,j} &:=& \left\{ L_{\psi }(b_{i}, b_{j}) \le b_{j}^{2}\psi(b_{j}^{-1}) \rho_{j}  \right\}, \\ 
	B_{j} &:=& B^{(\tau, M)}_{j} \cap B^{(L)}_{j}.
\end{eqnarray*}
According to 
\cite[Corollary~2]{RP81}, the following processes are independent
\begin{eqnarray*}
	W_{j}^{- } &:=& \left\{ W(t) \, | \, t \in [0, \tau(b_{j})] \right\}, 
\\ 
	W_{j}^{+} &:=& \left\{ W(\tau(b_{j})+t) - W(\tau(b_{j})) \, | \, t \ge 0 \right\}.
\end{eqnarray*} 
By definition, the independence of this processes means that the natural sigma-algebra of the process $W_{j}^{+}$ is independent of the sigma-algebra
\[
	\mathcal{F}_{\tau(b_{j})} := \sigma \left( \mathcal{F}_{\tau(b_{j})-\text{past}} \cup \mathcal{F}_{\tau(b_{j})-\text{pres}} \right),
\]
where (cf. \cite{wM76}) 
\begin{eqnarray*}
	\mathcal{F}_{\tau(b_{j})-\text{past}} &:=& \sigma \left( \left\{ A \cap \left\{ \tau(b_{j}) > t \right\} \, | \, A \in \sigma\big( W(t^{\prime}), \, t^{\prime} \in  [0, t]  \big) , t \ge 0  \right\} \right) , \\
	\mathcal{F}_{\tau(b_{j})-\text{pres}} &:=& \sigma \left( \left\{ A \cap \left\{ \tau(b_{j}) = S \right\} \, | \, A \in \mathcal{F}_{S}, S \text{ is a stopping time with respect to } W(\cdot )   \right\} \right) .
\end{eqnarray*} 
Here $\mathcal{F}_{S}$ denotes the standard sigma-algebra of $S$-past. Note that $\tau(b_{j})$ is not a stopping time.
\\
It follows that for $j \ge i \ge 1$,
\begin{eqnarray*}
	B_{j}^{(\tau, M)} &\in& \mathcal{F}(W_{j}^{+}) \cap \mathcal{F}_{\tau(b_{j+1})},
\\
	B_{i,j}^{(L)} &\in& \mathcal{F}(W_{i}^{+}) \cap \mathcal{F}_{\tau(b_{j})}.
\end{eqnarray*} 
Thus, for $j \ge i+1 \ge 1$, the events $B^{(L)}_{i}, B_{i}^{(\tau, M)}, B^{(L)}_{i+1, j}, B_{j}^{(\tau, M)}$ are independent.

Using the self-similarity of $W(\cdot )$, together with Lemma \ref{l:argmax}, we obtain
\[
	c_{\tau} := \mathbb{P} \left(B^{(\tau, M)}_{j}\right) = \mathbb{P} \left( \left\{ \argmax_{t \in [0, 1]} \, W\big(\tau(\lambda^{-1}) + t\big) = 0 \right\}  
\cap \left\{\tau(\bar{e}\lambda^{-1}) - \tau(\lambda^{-1}) > 1 \right\}  \right)   > 0,
\]
and $c_{\tau}$ does not depend on $j$.
By Lemma \ref {l:liminf_up}, applied with the function $\rho: b \mapsto g(b^2)$ and parameters $\xi_{\circ}, \xi^{\circ}$ as specified in the assumptions of Theorem \ref{t:liminf_up}, we have, for sufficiently large $j$, that
\[
     \mathbb{P}( B^{(L)}_{j} ) \ge \frac{c_{\varkappa ,5}}{2} \rho_{j}^{1 / (2-\varkappa)},
\] 
and hence
\begin{equation} \label{e:bj}
\mathbb{P} (B_{j}) =  c_{\tau} \cdot \mathbb{P}(B^{(L)}_{j})  \ge \frac{c_{\tau}c_{\varkappa,5}}{2} \rho_{j}^{1 / (2-\varkappa)}.
\end{equation} 
Since $g(\cdot )$ is a slowly varying function satisfying condition \eqref{cond:g2}, it follows that
\[
\sum_{j=1}^{\infty} \mathbb{P} (B_{j}) = \infty.   
\] 

To apply the generalization of the Borel--Cantelli lemma, namely the Kochen--Stone theorem (see \cite{KS64}), it suffices to show that
\[
\liminf\limits_{n \to \infty} \frac{\sum_{i,j=1}^{n} \mathbb{P} (B_{i} \cap B_{j}) }{\left( \sum_{j=1}^{n} \mathbb{P} (B_{j}) \right) ^2} \le 1.
\] 
For $j \ge i+1 \ge 2$,
\begin{comment}
using \eqref{e:intersection},
\end{comment}
we have
 \begin{equation}\label{e:bibj}
\mathbb{P} (B_{i} \cap B_{j}) \le \mathbb{P} \left( B^{(L)}_{i} \cap B_{i}^{(\tau, M)} \cap B^{(L)}_{i+1, j}\cap B_{j}^{(\tau, M)} \right) = c _{\tau}^2 \cdot  \mathbb{P} (B^{(L)}_{i})\cdot \mathbb{P} (B^{(L)}_{i+1,j}).       
\end{equation} 
Let us split the index set $(i,j) \in \N \times \N$ into two parts.
Define
\begin{equation*} \label{cond:ij}
	J^2 := \left\{ (i,j) \in \N \times \N \ | \ j \ge i+1 \ge 2, \ \ \frac{4+\xi_{\circ}}{2-\varkappa} \log \log b_{j}\ge (j-i-1)\log \bar{e}\right\}. 
\end{equation*}
It is true that
\[
   (i,j) \in J^2  \Longleftrightarrow b_{j} \ge b_{i+1} \ge \widehat{b_{j}} = b_{j}(\log b_{j})^{-(4+\xi_{\circ}) / (2-\varkappa)}.
\] 
By Lemma \ref {l:liminf_up}, applied with the function $\rho: b \mapsto g(b^2)$ and parameters $\xi_{\circ}, \xi^{\circ}$ as specified in the assumptions of Theorem \ref{t:liminf_up}, we have, for sufficiently large $j$ and $(i,j) \not\in J^2$, that
\begin{equation} \label{e:ij_out}
\mathbb{P} (B^{(L)}_{i+1, j}) \le (1+\delta)\mathbb{P} (B^{(L)}_{j}).
\end{equation} 

For all $(i,j) \in J^2$, we use the bound
\begin{equation} \label{e:Bij}
\mathbb{P} (B^{(L)}_{i+1,j}) \le \exp (-\mu_1(Q_{\psi }(b_{i+1}, b_{j}))),
\end{equation} 
where
\[
Q_{\psi }(b_{\circ}, b) := \left\{(s,z) \, | \, s \in (b_{\circ}/b,1], z \in (0,\infty),  z^{2} s^{2-\varkappa} \frac{\theta((sb)^{-1})}{\theta(b^{-1})} \ge \rho(b)\right\}.
\]
Fix $\delta > 0$. For $b \ge b_{\circ} \ge 1$, denote
\[
z_1(b_{\circ},b) :=  \sqrt{(1+\delta)\rho(b)} \left(\frac{b}{b_{\circ}}\right)^{1 - \varkappa / 2}.
\] 
For sufficiently large $b$, and for all $\sqrt{(1+\delta)\rho(b)} \le z \le \log b$ and $b_{\circ} \in [\,\widehat{b}, b]$, where $\widehat{b}$ is as defined in Lemma \ref{l:liminf_up}, the function $s_1(\cdot ,\cdot )$ defined by \eqref{def:s1} satisfies the analogue of inequality \eqref{e:s1}
\[
z \le z_1(b_{\circ},b)\Longleftrightarrow s_1(\rho(b), z) \ge \frac{b_{\circ}}{b}.
\]
It follows that, as $b \to \infty$, for all $b_{\circ} \in [\,\widehat{b}, b]$, we have, similarly to the calculation in \eqref{e:mu_b_est_low}, that
 \begin{eqnarray} \nonumber
	 \mu_1\left(Q_{\psi }(b_{\circ},b)\right) &=& \int_{0}^{\infty} \int_{b_{\circ} / b}^{1} 1_{\left\{ z^2s^{2-\varkappa}\frac{\theta((sb)^{-1})}{\theta(b^{-1})} \ge \rho(b) \right\} }   \mu_1(\mathrm{d}s \times \mathrm{d}z)
\\ &\ge&  \nonumber
\int_{\sqrt{(1+\delta)\rho(b)} }^{\min(z_1(b_{\circ},b),\log b)} \int_{s_1(\rho(b),z)}^{1} 2 \varphi (z) s ^{-1}  \mathrm{d}s \mathrm{d}z
+
\int^{\log b}_{\min(z_1(b_{\circ}, b), \, \log b)} \int_{b_{\circ} / b}^{1} 2 \varphi (z) s ^{-1}  \mathrm{d}s \mathrm{d}z
\\ &\ge& \label{e:mu_bb_2}
\frac{2}{2-\varkappa}|\log \rho(b)|\left(\frac{1}{2} - \overline{\Phi}(z_1(b_{\circ},b))\right) + c_{\varkappa,6} 
+ 2\log \frac{b}{b_{\circ}} \overline{\Phi}(z_1(b_{\circ}, b)) + o(1).
\end{eqnarray} 
Here $c_{\varkappa,6}:= \frac{1}{\sqrt{2\pi} (2-\varkappa)}\int_{0 }^{1} \zeta^{-1 / 2} e^{-\zeta / 2} \log \zeta  \, \mathrm{d} \zeta - \frac{2\log(1+\delta)}{2-\varkappa}< 0$, and $\overline{\Phi}(\cdot )$ denotes the tail distribution function of the standard normal law.
By condition \eqref{cond:g3}, for sufficiently large $j$, we have
\[
	|\log \rho(b_{j})| \ge \xi^{\circ} \log \log b_{j}.
\]
Combining this estimate with \eqref{e:mu_bb_2}, for sufficiently large $i$ and $(i,j) \in J^2$, we obtain
\begin{eqnarray*}
	\frac{\mu_1\left(Q_{\psi }(b_{i+1},b_{j})\right)}{\log \bar{e}} &\ge& 
	\frac{2\xi^{\circ}\log \log b_{j}}{(2-\varkappa)\log \bar{e}} \left(\frac{1}{2} - \overline{\Phi}(z_1(b_{i+1},b_{j}))\right) + 2(j-i-1) \overline{\Phi}(z_1(b_{i+1}, b_{j})) + c_{\varkappa,6}-1
	\\ &\ge&\left(\frac{\xi^{\circ}}{4+\xi_{\circ}}+\overline{\Phi}(z_1(b_{i+1}, b_{j}))\left(2-\frac{2\xi^{\circ}}{4+\xi_{\circ}}\right)\right) (j-i-1) + c_{\varkappa,6} - 1 
	\\ &\ge& \frac{\xi^{\circ}}{4+\xi_{\circ}}(j-i-1) + c_{\varkappa,6}-1.
\end{eqnarray*} 
By using \eqref{e:Bij}, for sufficiently large $i$, it holds that (with the summation taken over $j$)
\begin{equation} \label{e:ij_in}
\sum_{(i,j) \in J^2} \mathbb{P} (B^{(L)}_{i+1, j}) \le e^{-c_{\varkappa,6}+1} \sum_{k=0}^{\infty} \exp\left(-\frac{\xi^{\circ}\log \bar{e}}{4+\xi_{\circ}}k\right) =: c_{\varkappa,7} < \infty.
\end{equation} 

Thus, by combining \eqref{e:bibj} with estimates \eqref{e:ij_out} and \eqref{e:ij_in}, we obtain that there exists some $n_0 \in \mathbb{N}$ such that for all $n > n_0$, we have
\[
\sum_{i,j=1}^{n} \mathbb{P} (B_{i} \cap B_{j}) \le \sum_{i,j=1}^{n_0} \mathbb{P} (B_{i} \cap B_{j})+c_{\tau}^2c_{\varkappa,7}\sum_{i=1}^{n}  \mathbb{P} (B^{(L)}_{i}) + (1+\delta) c_{\tau}^2\sum_{i,j=1}^{n} \mathbb{P} (B^{(L)}_{i})\cdot \mathbb{P} (B^{(L)}_{j}).
\] 
By combining this estimate with equality \eqref{e:bj} and the fact that the series $\sum_{j} \mathbb{P} (B_{j})$ diverges, we obtain
\[
\liminf\limits_{n \to \infty} \frac{\sum_{i,j=1}^{n} \mathbb{P} (B_{i} \cap B_{j}) }{\left( \sum_{j=1}^{n} \mathbb{P} (B_{j}) \right) ^2} \le 1+\delta.
\] 
By letting $\delta \to 0$ and applying the Kochen--Stone theorem, we conclude that, with probability 1, the events $B_{j}$ occur infinitely often.

For all sufficiently large $j \ge 1$, we have
\[
	\left\{ \tau(b_{j}) \le (\lambda b_{j})^2 \right\}  \subset B^{(L)}_{j},
\] 
because if $B^{(L)}_{j}$ occurs but $\tau(b_{j}) > (\lambda b_{j})^2$, then by Proposition \ref{p:L_S} we obtain
\[
L_{\psi }(b_{j}) \ge  \tau(b_{j}) \psi (b_{j}^{-1}) > (\lambda b_{j})^2\psi (b_{j}^{-1}),
\]  
which contradicts $B^{(L)}_{j}$ for sufficiently large $j$.
Thus, the occurrence of the event $B_{j}$ implies that
\[
	\max_{t \in [\tau(b_{j}), (\lambda b_{j})^2]} W(t) = W(\tau(b_{j})).
\] 
If the event $B_j$ occurs, then by Proposition \ref{p:maj_form}, the solution to the problem
$|h|_{(\lambda b_{j})^2}^\psi\to \min, h\in M^{\prime}_{(\lambda b_{j})^2,r}$ is constant on the interval $[\tau(b_{j}), (\lambda b_{j})^2 ]$, which implies
$I_W^{\psi}((\lambda b_{j})^2,r) = I_W^{\psi}(\tau(b_{j}),r)$. From this, given that $B^{(L)}_{j}$ occurs, and by applying Proposition \ref{p:L_I}, it follows with probability 1 that
\[
	\liminf_{T \to \infty}\frac{I_{W}^{\psi}(T,r)}{T\psi(T^{-1 / 2})g(T)} \le \liminf_{j \to \infty}\frac{I_W^{\psi}((\lambda b_{j})^2,r)}{(\lambda b_{j})^2 \psi((\lambda b_{j})^{-1})\rho_{j}}
	\le \lambda^{\varkappa -2}\liminf_{\stackrel{j \to \infty}{B_j \text{ occurs  } }}\frac{L_{\psi }(b_{j})}{b_{j}^2\psi (b_{j}^{-1})\rho_{j}}
	\le \lambda^{\varkappa -2}.  
\] 
The first inequality follows from the property that $g(\cdot )$ is a slowly varying function.
This inequality also holds if $g(\cdot)$ is replaced by $p g(\cdot)$ for any $p > 0$, which concludes the proof of the theorem.

\subsection{Proof of Corollary \ref{cor:Hirsch}}
Consider $\psi(u) := u$ for $u > 0$. For a fixed $T > 0$, let $\chi_{*}(\cdot)$ denote the solution to the problem
$|h|_{T}^\psi \to \min, \; h \in M'_{T,r}$,
whose structure
is described in Proposition \ref{p:maj_form}. Since, for sufficiently large $T$, we have $\max_{0 \le t \le T} W(t) > r$, point (b) of that proposition applies; in particular $\chi_{*}(T) = \max_{0 \le t \le T} W(t)$. With probability 1, for sufficiently large $T$,
\[
I_W^{\psi}(T,r) = |\chi_{*}|_T^\psi = \int_{0}^{T} \chi_{*}^{\prime}(t) \mathrm{d}t = \chi_{*}(T) - \chi_{*}(0) = \max_{0 \le t \le T} W(t) - r.
\]
By applying Theorems \ref{t:liminf_low} and \ref{t:liminf_up} to the given functions $\psi(\cdot)$ and $g(\cdot)$, we obtain the desired result.

\subsection{Proof of Theorem \ref{t:limsup}}

{\bf Lower bound.}  
For a fixed $T > 0$, let $\chi_{*}(\cdot)$ denote the solution to the problem
$|h|_{T}^\psi \to \min, \; h \in M'_{T,r}$,
whose structure
is described in Proposition \ref{p:maj_form}. Then, by Jensen's inequality, we have
\[
I_{W}^{\psi}(T,r) = T \int_0^T \frac{\psi(\chi_{*}^{\prime}(t))}{T} \, \mathrm{d}t 
\ge T \psi \left(\int_0^T \frac{\chi_{*}^{\prime}(t)}{T} \, \mathrm{d}t \right) 
\ge T \psi\left(\frac{\chi_{*}(T) - r}{T}\right) 
\ge T \psi\left(\frac{W(T) - r}{T}\right).
\]
By the law of the iterated logarithm for $W(\cdot)$, and by the monotonicity of 
$\psi(\cdot)$ near zero, for any $\delta \in (0,1)$, with probability 1, it holds that
\[
\limsup_{T \to \infty} \frac{I_W^\psi(T,r)}{T \psi \left((1-\delta) \sqrt{\frac{2 \log \log T}{T}} \right)} \ge 1.
\]
It follows that, with probability 1,
\[
\limsup_{T \to \infty} \frac{I_W^\psi(T,r)}{T \psi \left(\sqrt{\frac{\log \log T}{T}} \right)} \ge (1-\delta)^{\varkappa} 2^{\varkappa / 2}.
\]
By letting $\delta \to 0$, we obtain the desired inequality.
\\ \\
{\bf Upper bound.}  
Fix $\delta > 0$.  
By the law of the iterated logarithm, there exists a random $T_{\circ} > 0$ such that for all $t > T_{\circ}$,
\[
W(t) < (1+\delta) \sqrt{2 t \log \log t}.
\]
Let $\chi_{*}(\cdot)$ denote the solution to the problem
$|h|_{T_{\circ}}^\psi \to \min, \; h \in M'_{T_{\circ}, r}$.  
For $t > 0$, consider the random function
\[
h(t) := 
\begin{cases}
\max\left(\chi_{*}(t), (1+\delta) \sqrt{2 t \log \log t} \right), & t \in [0, T_{\circ}], \\
(1+\delta) \sqrt{2 t \log \log t}, & t > T_{\circ}.
\end{cases}
\]
For any $T > 0$, the restriction of $h(\cdot)$ to $[0,T]$ is an admissible approximation of $W(\cdot)$ on this interval. Therefore, with probability 1,
\[
I_W^\psi(T,r) \le |h|_T^\psi.
\]
With probability 1, as $t \to \infty$, we have
\[
h'(t) = (1+\delta) \sqrt{\frac{\log \log t}{2 t}} (1 + o(1)).
\]
For any $\delta_1 > 0$, there exists a random time $t_{\circ} > T_{\circ}$, independent of $T$, such that for all $t_{\circ} < t \le T$, as $T \to \infty$, it holds that
\begin{eqnarray*}
\frac{\theta(h'(t))}{\theta \left(\sqrt{\frac{\log \log T}{2 T}} \right)} &\le& (1+\delta_1) \exp \left( - \int_{\sqrt{\frac{\log \log T}{2 T}}}^{\sqrt{\frac{\log \log t}{2 t}}} \frac{\eps(v)}{v} \mathrm{d}v \right) \\
&\le& (1+\delta_1) \exp \left( \frac{(\log T)^{1-\alpha} - (\log t)^{1-\alpha}}{(1-\alpha) 2^{1-\alpha}} (1 + o(1)) \right) \\
&=& (1+\delta_1) \exp \left( O \left( \frac{\log \frac{T}{t}}{(\log T)^\alpha} \right) \right).
\end{eqnarray*}
Therefore, with probability 1, we obtain
\begin{eqnarray*}
\limsup_{T \to \infty} \frac{|h|_T^\psi}{T \psi \left(\sqrt{\frac{\log \log T}{2 T}}\right)} &=& \limsup_{T \to \infty} \int_{t_{\circ}}^{T} \frac{\psi(h'(t))}{T \psi \left(\sqrt{\frac{\log \log T}{2 T}} \right)} \, \mathrm{d}t \\
&\le& (1+\delta)^{\varkappa} \limsup_{T \to \infty} \int_{t_{\circ}}^{T} \left(\frac{t}{T}\right)^{-\varkappa/2} \frac{\theta(h'(t))}{\theta \left(\sqrt{\frac{\log \log T}{2 T}} \right)} \frac{\mathrm{d}t}{T} \\
&=& (1+\delta)^{\varkappa} \limsup_{T \to \infty} \int_{t_{\circ} T^{-1}}^{1} \eta^{-\varkappa/2} (1+\delta_1) \exp \left( O\left(\frac{|\log \eta|}{(\log T)^\alpha}\right) \right) \, \mathrm{d}\eta \\
&=& (1+\delta)^{\varkappa} \limsup_{T \to \infty} \frac{1+\delta_1}{1 - \varkappa/2 - O((\log T)^{-\alpha})} \\
&=& \frac{(1+\delta)^{\varkappa} (1+\delta_1)}{1 - \varkappa/2}.
\end{eqnarray*}
By taking $\delta$ and $\delta_1$ to zero successively, we obtain the desired inequality.

\section{Proofs of Theorems for $\varkappa \ge 2$} \label{sec:4}
\subsection{Preliminaries}

In the following lemma, we begin to describe the conditions under which $m_{\psi}(\cdot)$ serves as a proper normalization for $L_{\psi}(\cdot)$ to converge almost surely to a constant.

\begin{lem} \label{l:as}
Define the function  
\[
q(b) := \frac{\mathbb{D}L_\psi(b)}{(\mathbb{E} L_{\psi}(b))^{2}}.
\]
Then, if there exists some $\sigma > 0$ such that, as $b \to \infty$, 
\begin{equation} \label{cond:log}
       q(b) =  O\left( (\log b)^{-\sigma} \right), 
\end{equation}
and, for some $\gamma > \frac{1}{\sigma}$, it holds that
\begin{equation} \label{cond:m}
	\lim_{n \to \infty} \frac{\mathbb{E} L_{\psi}\left(\exp(n^{\gamma})\right)} 
        {\mathbb{E} L_\psi(\exp((n+1)^{\gamma}))} = 1,
\end{equation}
then
\[
     \lim\limits_{b \to \infty} \frac{L_\psi(b)}{\mathbb{E} L_{\psi}(b)} \stackrel{\text{a.s.}}{=} 1.
\]
\end{lem}

\begin{proof}
The function $q(\cdot)$ is well-defined due to representation \eqref{e:L_int}.
For any $\delta > 0$ and $b \ge 1$, define the sets  
\[
    D(b) := \left\{ \left\lvert \frac{L_\psi(b)}{\mathbb{E} L_{\psi}(b)} -1 \right\rvert > \delta \right\}.
\]
By Chebyshev's inequality, we have
\[
\mathbb{P}(D(b)) \le \frac{q(b)}{\delta^{2}}.
\]
Let $b_{n} := \exp(n^{\gamma})$. From $\gamma > \frac{1}{\sigma}$ and condition \eqref{cond:log}, it follows that
\[
\sum_{n=1}^{\infty} \mathbb{P}(D(b_{n})) < \infty.
\]
This implies, by the Borel--Cantelli lemma, that
\begin{equation} \label{e:as_seq}
\lim\limits_{n \to \infty} \frac{L_\psi(b_{n})}{\mathbb{E} L_\psi(b_{n})} \stackrel{\text{a.s.}}{=} 1.
\end{equation}
Let $b \in [b_{n}, b_{n+1}]$. Since $L_\psi(\cdot)$ and $m_{\psi}(\cdot)$ are non-decreasing, with probability 1 it holds that 
\begin{equation} \label{e:lr_est}
\frac{L_\psi(b_{n+1})}{\mathbb{E} L_\psi(b_{n})} \ge \frac{L_\psi(b)}{\mathbb{E} L_{\psi}(b)} \ge \frac{L_\psi(b_{n})}{\mathbb{E} L_{\psi}(b_{n+1})}.
\end{equation}
By condition \eqref{cond:m}, 
\[
\lim\limits_{n \to \infty} \frac{\mathbb{E} L_\psi(b_{n})}{\mathbb{E} L_{\psi}(b_{n+1})} = 1.
\]
Therefore, by \eqref{e:as_seq}, the right-hand side and left-hand side in \eqref{e:lr_est} converge to 1 almost surely.
\end{proof}

Now we describe sufficient conditions for the almost sure convergence of 
$\frac{L_{\psi}(\cdot)}{m_{\psi}(\cdot)}$ to a constant, stated in terms of the function $\psi(\cdot)$.

\begin{lem} \label{l:start}
Let $\psi(\cdot)$ be a convex function of the form \eqref{cond:reg} with $\varkappa = 2$ and some function $\theta(\cdot)$, which is not necessarily slowly varying. Suppose further that $\theta(\cdot)$ satisfies condition \eqref{cond:theta} and either condition \eqref{cond:alpha} or \eqref{cond:alpha_alt}.  
Then
\begin{equation*} \label{e:asw}
	\lim\limits_{b \to \infty} \frac{L_\psi(b)}{m_\psi(b)} \stackrel{\text{a.s.}}{=} 1.
\end{equation*}
\end{lem}

\begin{proof}
By using the representation \eqref{e:L_int}, for $b > b_{\circ} = 1$, we observe that the expectation of $L_{\psi}(b)$ is exactly the quantity $m_{\psi}(b)$
\begin{eqnarray} \nonumber
	\mathbb{E}L_{\psi}(b) &=& \int\limits_{1}^{b} \int\limits_{0}^{\infty} l \psi(a^{-1}) \, 
(a^{2}\sqrt{l})^{-1}\varphi(\sqrt{l}/a)\, \mathrm{d}a \, \mathrm{d}l 
\\ \nonumber  
			      &=& \int\limits_{1}^{b} \int\limits_{0}^{\infty} 2z^{2}a \psi(a^{-1})\varphi(z) \, 
\mathrm{d}a \, \mathrm{d}z 
\\ \label{e:E_m}
			      &=& \int\limits_{1}^{b} a \psi(a^{-1}) \, \mathrm{d}a.
\end{eqnarray}
For the variance of $L_{\psi}(b)$ with $b \ge 1$, we have
\begin{equation*}
\mathbb{D}L_{\psi}(b) = 3\int_{1}^{b} a^{3} \psi(a^{-1})^{2} \, \mathrm{d}a.
\end{equation*}

If condition \eqref{cond:alpha} holds, then by using condition \eqref{cond:theta}, as $b \to \infty$, we have
\[
q(b) = O \left(\frac{\left( \log b \right)^{2\alpha^{\diamond} + 1} }{(\log b)^{2(\alpha_\diamond+1)}}\right) = O\left( \left( \log b \right)^{2\alpha^{\diamond} - 2\alpha_\diamond - 1} \right).  
\]
Let us verify condition \eqref{cond:m} for $\gamma = \frac{1}{-2\alpha^{\diamond} + 2\alpha_\diamond + 1} + \delta$ with some $\delta > 0$. Let $b_{n} := \exp(n^{\gamma})$, then, as $n \to \infty$, it holds that
\begin{eqnarray} \label{e:m_m}
1 \le \frac{m_\psi(b_{n+1})}{m_{\psi}(b_{n})} 
&=&  
1 + O\left( \frac{((\log b_{n+1})^{\alpha^{\diamond}+1} - (\log b_{n})^{\alpha^{\diamond}+1})}
{(\log b_{n})^{\alpha_\diamond+1}} \right) 
\\ &=& \nonumber
1 + O\left( \frac{(n+1)^{(\alpha^\diamond+1)\gamma} - n^{(\alpha^{\diamond}+1)\gamma}}{n^{(\alpha_\diamond+1)\gamma}} \right)  
\\ &=& \nonumber
1 + O\left(n^{(\alpha^{\diamond}+1)\gamma - 1 - (\alpha_\diamond+1)\gamma}\right) = 1 + o(1).
\end{eqnarray}
The last equality holds for sufficiently small $\delta > 0$ due to condition \eqref{cond:alpha}.

If condition \eqref{cond:alpha_alt} holds, then for any $\delta_1 > 0$, as $b \to \infty$, we have
\[
q(b) = O\left((\log b)^{-2\alpha_\diamond - 2 + \delta_1}\right).
\]
Condition \eqref{cond:m} is verified similarly for $\gamma = \frac{1}{2\alpha_\diamond + 2} + \delta$ with sufficiently small $\delta > 0$.

In both cases, the proof is completed by applying Lemma \ref{l:as}.
\end{proof}

\subsection{Proof of Theorem \ref{t:finite}}
Due to equality \eqref{e:E_m}, the uniform boundedness of $m_\psi(\cdot)$ implies that $\mathbb{E} L_{\psi}(\infty) < \infty$. This latter condition suffices to ensure that the Poisson integral $L_{\psi}(\infty)$ is well-defined as an almost surely finite random variable. Since $\tau(\cdot)$ is almost surely unbounded, Proposition \ref{p:L_I} implies the almost sure boundedness of $I_{W}^{\psi}(\cdot, r)$.  
Because $I_{W}^{\psi}(\cdot, r)$ is monotonically increasing, it converges almost surely to a finite random variable.

\subsection{Proof of Theorem \ref{t:infinite}}
Using Lemma \ref{l:tau_est}, with probability 1, for sufficiently large $T$, we have
\[
\tau(b_{T}^{-}) := \tau\left(T^{1/2}(\log \log T)^{-1}\right) < T < \tau\left(T^{1/2}(\log T)^{4}\right) =: \tau(b_{T}^{+}).
\]
By combining Proposition \ref{p:L_I} with condition \eqref{cond:infinite} and the monotonicity of $L_{\psi}(\cdot)$, we obtain
\begin{eqnarray*}
	\liminf\limits_{T \to \infty} \frac{I_W^{\psi}(T,r)}{m_\psi(T^{1/2})} 
	&\ge& \liminf\limits_{T \to \infty} 
	\frac{L_{\psi}(b_{T}^{-})}{m_\psi(b_{T}^{-})} 
	\frac{m_\psi(b_{T}^{-})}{m_{\psi}(T^{1/2})}, 
	\\             
	\limsup\limits_{T \to \infty} \frac{I_W^{\psi}(T,r)}{m_\psi(T^{1/2})} 
	&\le& \limsup\limits_{T \to \infty} 
	\frac{L_{\psi}(b_{T}^{+})}{m_\psi(b_{T}^{+})} 
	\frac{m_\psi(b_{T}^{+})}{m_{\psi}(T^{1/2})}.
\end{eqnarray*}
Taking into account Lemma \ref{l:start}, it holds that
\begin{equation} \label{e:lr_est2}
\liminf\limits_{T \to \infty} \frac{m_\psi(b_{T}^{-})}{m_{\psi}(T^{1/2})} \le 
\liminf\limits_{T \to \infty} \frac{I_W^{\psi}(T,r)}{m_\psi(T^{1/2})} 
\le
\limsup\limits_{T \to \infty} \frac{I_W^{\psi}(T,r)}{m_\psi(T^{1/2})} 
\le
\limsup\limits_{T \to \infty} \frac{m_\psi(b_{T}^{+})}{m_{\psi}(T^{1/2})}.
\end{equation}
Similarly to the computation in \eqref{e:m_m}, we have
\begin{eqnarray*}
	\frac{m_\psi(b_{T}^{+})}{m_{\psi}(T^{1/2})} &\le& 1 + O\left( \frac{(\log b_{T}^{+})^{\alpha^{\diamond}+1} - (\log (T^{1/2}))^{\alpha^{\diamond}+1}}{(\log (T^{1/2}))^{\alpha_\diamond+1}} \right)
	\\ &=&
	1 + O\left( (\log T)^{\alpha^{\diamond} - \alpha_\diamond - 1} \log \log T \right) = 1 + o(1).
\end{eqnarray*}
The lower bound for $\frac{m_\psi(b_{T}^{-})}{m_{\psi}(T^{1/2})}$ is obtained similarly. Substituting these bounds into inequality \eqref{e:lr_est2}, we obtain the desired result.

\bigskip
The author is grateful to M.A. Lifshits for setting the problem and for valuable assistance with this work.

\end{document}